\documentclass[a4paper, 11pt, reqno]{amsart}
\usepackage[utf8]{inputenc}
\usepackage[english]{babel}
\usepackage{amsmath}
\usepackage{amsthm}
\usepackage{tikz}
\usetikzlibrary{svg.path}
\usetikzlibrary{shapes,arrows,chains}
\usepackage{amsfonts}
\usepackage{graphicx}
\usepackage{xcolor}
\usepackage{bbm}
\usepackage{csquotes}
\usepackage{latexsym}
\usepackage{float}
\usepackage{hyperref}
\usepackage[backend=biber, style=alphabetic, sorting=nyt]{biblatex}
\usepackage{amssymb}
\usepackage{mathtools}
\usepackage{enumitem}
\usepackage{calc}
\usepackage{setspace}
\usepackage{graphicx}
\usepackage{tikz-cd}
\usepackage{mathscinet}
\usepackage{mathrsfs}

\newtheorem{thm}{Theorem}[section]
\newtheorem*{thm*}{Theorem}
\newtheorem{lem}[thm]{Lemma}

\newtheorem{prop}[thm]{Proposition}
\newtheorem{coro}[thm]{Corollary}

\theoremstyle{definition}
\newtheorem{defn}[thm]{Definition}

\newtheorem{rem}[thm]{Remark}

\newtheorem{nota}[thm]{Notation}
\numberwithin{equation}{section}

\newcommand{\codim}{\textnormal{codim}}

\renewcommand{\exp}{\textnormal{exp}}

\renewcommand{\o}[1]{\overline{ #1 }}

\newcommand{\gl}{\textnormal{GL}}

\renewcommand{\epsilon}{\varepsilon}
\renewcommand{\phi}{\varphi}

\newcommand{\stab}{\textnormal{Stab}}

\newcommand{\trop}{\textnormal{Trop}}

\newcommand{\aff}{\textnormal{aff}}
\newcommand{\Log}{\textnormal{Log}}
\newcommand{\am}{\mathcal{A}_W}
\newcommand{\ctimes}{\mathbb{C}^\times}

\newcommand{\ini}{\textnormal{in}}
\newcommand{\relint}{\textnormal{relint}}
\renewcommand{\star}{\textnormal{star}}

\newcommand{\supp}{\textnormal{Supp}}
\newcommand{\val}{\textnormal{val}}

\newcommand{\res}{\textnormal{res}}

\newcommand{\lin}{\textnormal{lin}}

\begin{document}

\title{Likely intersections in powers of the multiplicative group}

\author[G. A. Dill]{Gabriel A. Dill}
\address[G. A. Dill]{Institut de Math\'ematiques, Universit\'e de Neuch\^atel, Rue Emile-Argand 11, 2000 Neuch\^atel, Switzerland}
\email{gabriel.dill@unine.ch}

\author[F. Gallinaro]{Francesco Gallinaro}
\address[F. Gallinaro]{Dipartimento di Matematica, Universit\`a di Pisa, Largo Bruno Pontecorvo 5, 56127, Pisa, Italy}
\email{francesco.gallinaro@dm.unipi.it}

\date{\today}

\date{\today}

\begin{abstract}
	We derive two finiteness properties as consequences of the \emph{geometrical non-degeneracy} of an algebraic subvariety $W$ of a power of the multiplicative group, concerning the intersections of $W$ with translates of a subtorus $H$ of dimension greater than or equal to the codimension of $W$. The first one is that every translate of $H$ intersects $W$, unless $H$ is contained in one of finitely many proper subtori depending only on $W$. The second one is that every translate of $H$ by a torsion point intersects $W$, unless the translate is contained in one of finitely many proper algebraic subgroups, again depending only on $W$. We use methods from tropical geometry and equidistribution, as well as some very mild model theory.
\end{abstract}

\maketitle

\tableofcontents
\newpage
\section{Introduction}

Let $K$ be an algebraically closed field of characteristic $0$. We denote by $\mathbb{G}_{m,K}$ the multiplicative group over $K$ and identify $\mathbb{G}_{m,K}^n(K)$ with $(K^\times)^n$ as usual. An irreducible algebraic subvariety $W$ of $\mathbb{G}_{m,K}^n$ is called \textit{geometrically non-degenerate} if for every algebraic subgroup $J$ of $\mathbb{G}_{m,K}^n$, with associated algebraic quotient map $\pi_J:\mathbb{G}_{m,K}^n \twoheadrightarrow \mathbb{G}_{m,K}^n/J$, we have that $\dim \pi_J(W)=\min\{\dim W, n-\dim J\}$. For example, it is easy to see that if $W$ is a curve, then it is geometrically non-degenerate if and only if it is not contained in any coset of a proper algebraic subgroup of $\mathbb{G}_{m,K}^n$; and that if $W$ is a hypersurface, then it is geometrically non-degenerate if and only if there is no infinite algebraic subgroup $J$ of $\mathbb{G}_{m,K}^n$ such that $J \cdot W=W$. 

We now assume that $K = \mathbb{C}$ and write $\mathbb{G}^n_m$ instead of $\mathbb{G}^n_{m,\mathbb{C}}$. Given equations defining $W$, the dimension condition needs to be checked for only finitely many algebraic subgroups $J$ and these can be determined effectively; see \cite[Theorem 1.4]{BMZ07} and \cite[Sections 1.3.3 and 1.3.4]{ZannierBook}. One can also check that $W$ is geometrically non-degenerate if and only if for every subtorus $J$ of $\mathbb{G}_{m}^n$ of dimension $n - \dim W$ we have that $\dim \pi_J(W)=\dim W$. Note that, since $\pi_J(W)$ is in general just constructible and not necessarily Zariski closed, this does not imply that $W$ intersects every translate of $J$. Let us remark that the notion of geometrical non-degeneracy is not new: to the best of our knowledge, it made its first appearance in 1981 in \cite{Ran1980/81} (for subvarieties of abelian varieties, but the translation to $\mathbb{G}_m^n$ is immediate).

In this paper, we prove the following results on geometrically non-degenerate subvarieties of $\mathbb{G}_m^n$.

\begin{thm}\label{thm:main}
	Let $W \subseteq \mathbb{G}^n_{m}$ be an irreducible geometrically non-degenerate subvariety.
	
	Then there exists a finite set $\mathcal{H} = \{H_1,\hdots,H_N\}$ such that $H_i \subsetneq \mathbb{G}^n_{m}$ is a subtorus for all $H_i \in \mathcal{H}$ and such that for every subtorus $H \subseteq \mathbb{G}^n_{m}$ with $\dim H + \dim W \geq n$, one of the following holds:
	\begin{enumerate}
		\item[(i)] For every $z \in (\ctimes)^n$, $W(\mathbb{C}) \cap z \cdot H(\mathbb{C}) \neq \emptyset$ or
		\item[(ii)] $H \subseteq H_i$ for some $H_i \in \mathcal{H}$.
	\end{enumerate}
\end{thm}

\begin{thm}\label{thm:main-two}
	Let $W \subseteq \mathbb{G}^n_{m}$ be an irreducible geometrically non-degenerate subvariety.
	
	Then there exists a finite set $\mathcal{G} = \{G_1,\hdots,G_N\}$ such that $G_i \subsetneq \mathbb{G}^n_{m}$ is an algebraic subgroup for all $G_i \in \mathcal{G}$ and such that for every subtorus $H \subseteq \mathbb{G}^n_{m}$ with $\dim H + \dim W \geq n$ and for every torsion point $\zeta \in \mathbb{G}^n_{m}(\mathbb{C})$, one of the following holds:
	\begin{enumerate}
		\item[(i)] $W(\mathbb{C}) \cap \zeta \cdot H(\mathbb{C}) \neq \emptyset$ or
		\item[(ii)] $\zeta \cdot H \subseteq G_i$ for some $G_i \in \mathcal{G}$.
	\end{enumerate}
\end{thm}

The interest of these results comes from their connections with two areas of very active research which sit at the border between model theory and Diophantine geometry, namely \textit{unlikely intersections} and \textit{exponential algebraic-closedness}. 

\subsection{Unlikely (and likely) intersections}

In the seminal paper \cite{BMZ99}, Bombieri, Masser, and Zannier proved that the intersection of a curve in an algebraic torus (i.e. a power of the multiplicative group) with the union of all algebraic subgroups of codimension $2$ is finite provided that the curve is defined over the algebraic numbers and not contained in a coset of a proper subtorus. Intuitively, the intersection of a curve with an algebraic subgroup of codimension at least $2$ should usually be empty for dimension reasons and the result of Bombieri, Masser, and Zannier states that, under the above-mentioned hypotheses on the curve, this holds true apart from at most finitely many exceptional points.

Since 1999, there has been a lot of research on generalizing the result of Bombieri, Masser, and Zannier in various ways: replacing the curve by a higher-dimensional subvariety, weakening the hypotheses on the curve, or replacing the ambient algebraic torus by another ambient variety with a collection of ``special" subvarieties similar to the collection of torsion cosets, i.e. of irreducible components of algebraic subgroups of the torus.

In particular, it follows from \cite{Maurin08} and \cite{BMZ08} that the result of Bombieri, Masser, and Zannier continues to hold for a curve defined over $\mathbb{C}$ that is not contained in any proper algebraic subgroup of the torus and one can show that this last hypothesis is actually necessary.

Such results are collected under the umbrella term of ``unlikely intersections" since they concern intersections that are expected to be empty for dimension reasons. Unlikely intersections have also been studied in abelian varieties and semiabelian varieties as well as in pure and mixed Shimura varieties. All of these ambient varieties possess a collection $\mathcal{S}$ of special subvarieties $S$. It turns out that, in order to understand unlikely intersections, it is advantageous to look at the larger class of intersections of atypical dimension: for a fixed subvariety $W$ and varying $S \in \mathcal{S}$, we call a component $A$ of an intersection $W \cap S$ atypical if $\dim A > \dim W - \codim S$. The intersection is unlikely if $\dim W - \codim S < 0$.

General conjectures, due to Zilber \cite{Zil02} (for semiabelian varieties), Pink \cite{PinkUnpubl} (for mixed Shimura varieties), and Bombieri, Masser, and Zannier \cite{BMZ07} (for algebraic tori), predict that unlikely or atypical intersections are rare in one way or another. Pink's and Zilber's conjectures are slightly different, but actually equivalent in many cases (see \cite[Section 12]{BD2}). In the case of algebraic tori, which is the relevant one for the purposes of this article, Bombieri, Masser, and Zannier have shown in the Appendix to \cite{BMZ08} that all formulations of the conjecture are equivalent. The Zilberian formulation is equivalent to the statement that a given variety has at most finitely many atypical subvarieties that are maximal with this property with respect to inclusion. There is an equivalent formulation that is more similar to our theorem: every atypical subvariety is contained in one of finitely many proper algebraic subgroups of the algebraic torus. The analogy is not perfect: the proper algebraic subgroups contain the atypical intersections and not the torsion cosets with which one intersects, and indeed the latter would be impossible.

Typically, the known results in the direction of these conjectures are confined to cases where either the dimension or the codimension is low (curves, hypersurfaces, subvarieties of codimension $2$). However, more is known for geometrically non-degenerate subvarieties of algebraic tori and abelian varieties, see \cite[Corollary 1.5]{Habegger_2009}, \cite[Th\'eor\`eme 1.1]{Maurin11}, and \cite[Theorem 9.15(iv)]{HP16}.

More recently, people have started to also consider likely intersections, usually proving that the union of a certain class of likely intersections in a given irreducible subvariety is (analytically or Zariski) dense in that subvariety (of course, the intersection with the whole ambient variety is always likely, so this one will always be excluded from the class under consideration -- the most natural option is to consider special subvarieties whose codimension is equal to the dimension of the fixed subvariety). See \cite{ACZ20, G20, G21Corr, AsvinG, EterovicScanlon, TT23, BKU24}. See also the earlier \cite[Theorem 1.2]{Habegger_2010} and \cite[Remark 3.4.5]{ZannierBook} for likely intersections in the square of the modular curve $Y(1)$.

For an irreducible subvariety $W$ of an algebraic torus $\mathbb{G}_m^n$, Zariski density in $W$ of the intersection of $W$ with the union of all algebraic subgroups of codimension $\dim W$ is easy to see: choose $\dim W$ algebraically independent monomials that remain algebraically independent when restricted to $W$ and use that the image of the induced morphism from $W$ to $\mathbb{G}^{\dim W}_m$ is constructible and Zariski dense. Note that analytic density may fail: for example, the set of torsion points is not analytically dense in $\mathbb{G}_m^n$.

In this article, our goal is slightly different: we want to show that, under suitable hypotheses, almost every intersection that one expects to be non-empty for dimension reasons is actually non-empty. More concretely, for a geometrically non-degenerate subvariety $W$ of an algebraic torus over $\mathbb{C}$, we show in Theorem \ref{thm:main-two} that, if $W \cap T = \emptyset$ for some torsion coset $T$ of codimension at most $\dim W$, then $T$ is contained in one of finitely many proper torsion cosets. We also show in Theorem \ref{thm:main} that, if $W \cap zH = \emptyset$ for some subtorus $H$ of codimension at most $\dim W$ and some $\mathbb{C}$-point $z$, then $H$ is contained in one of finitely many proper subtori. Here, one cannot expect $zH$ to be contained in one of finitely many translates of proper subtori (unless $W$ is a curve, see Corollary \ref{cor:curvecase} below): take for example $W$ to be the geometrically non-degenerate hypersurface that is defined inside $\mathbb{G}_m^3$ by the equation $x + y + z = 1$. Then, for any $c \in \mathbb{C} \setminus \{0,1\}$, the intersection of $W$ with the coset $\{(1-c,c)\} \times \mathbb{G}_m$ is empty, but the union of all these cosets is not contained in the union of finitely many cosets of proper subtori. Note that this does not contradict our statement about torsion cosets since $c$ and $1-c$ are simultaneously roots of unity if and only if $c$ is a primitive sixth root of unity. The connection to unlikely intersections that becomes visible here will be explained in more detail later. This example also shows that Theorem \ref{thm:main-two} becomes false over the algebraic closure of a finite field since then all the points $(1-c,c)$ ($c \neq 0, 1$) are torsion points. We do not know whether Theorem \ref{thm:main} holds in positive characteristic or not.

The hypothesis of geometrical non-degeneracy is necessary for a result such as Theorem \ref{thm:main-two} to hold: if $\dim \pi_H(W) < \dim W$ for some subtorus $H$ of dimension $n - \dim W$, then $\zeta H \cap W = \emptyset$ for a Zariski dense set of torsion cosets $\zeta H$. Since geometrical non-degeneracy is not inherited by irreducible subvarieties, we can unfortunately not conclude that every unlikely non-intersection is explained by one of finitely many unlikely non-intersections. Furthermore, such a statement would just be plain false: for example, the hypersurface $W_0$ defined by the equation $x + 2y + z = 1$ in $\mathbb{G}_m^3$ has trivial stabilizer and is therefore geometrically non-degenerate, but has empty intersection with any torsion coset defined by a pair of equations $z = 1$ and $x = \zeta y$ for a root of unity $\zeta$. But any algebraic subgroup that contains infinitely many of these torsion cosets will also contain the algebraic subgroup $H$ defined by the equation $z = 1$ and the intersection of $H$ with $W_0$ is non-empty, irreducible, and geometrically degenerate as a subvariety of $H$.

As for Theorem \ref{thm:main}, a $1$-dimensional subtorus of $\mathbb{G}_m^2$ satisfies its conclusion even though it is geometrically degenerate; on the other hand, given a $1$-dimensional subtorus $H$ of $\mathbb{G}_m^3$, there are infinitely many distinct $2$-dimensional subtori which have cosets not intersecting $H$, so the conclusion is not true for arbitrary subvarieties. 

The analogues of Theorems \ref{thm:main} and \ref{thm:main-two} trivially hold if we replace $\mathbb{G}_m^n$ by a complex abelian variety $A$: if $W \subseteq A$ is geometrically non-degenerate and $B$ is an abelian subvariety of $A$ such that $\dim B + \dim W \geq \dim A$, then the image of $W$ under the quotient homomorphism $A \twoheadrightarrow A/B$ is Zariski dense and closed, so equal to $A/B$. It is a natural question whether this can be generalized further to the case of semiabelian varieties. We leave this question to future work. For now, let us just note that the analogous statement for $\mathbb{G}_a^n$ is false: the irreducible surface defined by $x^3 + y^3 + z^3 = 1$ in $\mathbb{G}_a^3$ has trivial stabilizer, but its intersection with the line of slope $(a,b,c) \neq (0,0,0)$ through $(0,0,0)$ is empty as soon as $a^3 + b^3 + c^3 = 0$. The union of all such lines is an irreducible cubic surface and therefore cannot be contained in any finite union of planes.

Such results can be regarded as geometric analogues of the arithmetic results obtained in \cite{BIR_2008, Grant_Ih_2013}; the results in \cite{BIR_2008} were the first-named author's original inspiration for this work. Several other recent results can be interpreted as density statements for certain classes of likely intersections in arithmetic schemes, e.g. \cite{Charles18, Shankar_Tang_2020, SSTT22}.

Results like Theorem \ref{thm:main-two} are simultaneously weaker and stronger than the density of some class of likely intersections: weaker since there might not exist enough algebraic subgroups to produce a dense union of a non-trivial class of likely intersections (consider for example a proper positive-dimensional subvariety of a simple abelian variety), but stronger since mere density can be achieved even if most likely intersections are empty.

\subsection{Exponential-Algebraic Closedness}

In his work on the model theory of the complex exponential function \cite{Zil05}, Zilber conjectured a Nullstellensatz-type property for systems of exponential-polynomial equations over the complex numbers: he predicted that all such systems which are solvable in certain exponential fields extending $\mathbb{C}$ should already be solvable in $\mathbb{C}$. This conjecture, now known as the \textit{Exponential-Algebraic Closedness conjecture} (EAC conjecture for short), is formulated geometrically: it states that all irreducible algebraic subvarieties of $\mathbb{G}_a^n \times \mathbb{G}_m^n$ which satisfy the geometric conditions of \textit{freeness} and \textit{rotundity} should contain a point of the form $(z,\exp(z))$, where $\mathbb{G}_a$ denotes the additive group over $\mathbb{C}$ and $\exp$ is the $n$-th Cartesian power of the complex exponential map. An algebraic subvariety $V \subseteq \mathbb{G}_a^n \times \mathbb{G}_m^n$ is called \emph{free} if no translate of its projection to $\mathbb{G}_a^n$ is contained in a proper linear subspace of $\mathbb{G}_a^n$ defined over $\mathbb{Q}$ and its projection to $\mathbb{G}_m^n$ is not contained in a coset of a proper algebraic subgroup of $\mathbb{G}_m^n$. The definition of a \emph{rotund} algebraic subvariety of $\mathbb{G}_a^n \times \mathbb{G}_m^n$ is given in Subsection \ref{subsec:rotund}; it is a condition stating that the images of $V$ under certain quotient maps satisfy appropriate dimension inequalities, for example $\dim V \geq n$.

The EAC conjecture has attracted considerable interest from people working in model theory and Diophantine geometry, and started a line of research which goes in a similar direction to that of likely intersection problems, aiming to prove that suitable conditions are sufficient to ensure the non-emptyness of the intersection between some algebraic varieties and the graphs of certain analytic functions, such as the complex exponential \cite{BM,AKM,Gal23,MM}, the exponential maps of abelian varieties \cite{Gal23b}, the modular $j$-function and other uniformizations of Shimura varieties \cite{EH,Gal,AEM, EZ} and the $\Gamma$ function \cite{EP}. An important difference is that while likely intersection problems are usually seen as problems concerning the interaction between the arithmetic and the algebro-geometric structure of certain algebraic varieties, the analytic structure plays an important role in EAC-type problems.

In \cite{Gal23}, the second-named author of the present paper established the EAC conjecture for algebraic subvarieties of $\mathbb{G}_a^n \times \mathbb{G}_m^n$ which split as a product of the form $L \times W$, where $L \leq \mathbb{G}_a^n$ is a linear subspace and $W \subseteq \mathbb{G}_m^n$ is an algebraic subvariety. The proofs employed methods from \emph{tropical geometry}, a relatively young branch of mathematics which has been described as a ``combinatorial shadow'' of algebraic geometry. The connection between EAC-type problems and this kind of methods had already been observed by Zilber in his early work on the conjecture \cite{Zil02}, but the development of tropical geometry which has taken place in the last few decades allowed to go well beyond Zilber's results. In this paper, we take this connection forward: we use tropical methods to prove results with a similar flavour to EAC-type questions, but which are purely algebraic.

Studying the intersection between an algebraic subvariety $W$ of $\mathbb{G}_m^n$ and the cosets of an algebraic subtorus $J$ is the same as studying the intersection of the graph of the exponential with algebraic varieties of the form $(z+L) \times W$, where $z \in \mathbb{G}_a^n(\mathbb{C})$ and $L$ is a linear subspace of $\mathbb{G}_a^n$ defined over $\mathbb{Q}$ such that $\exp(L(\mathbb{C}))=J(\mathbb{C})$. Of course, such an algebraic variety will never satisfy the definition of freeness if $J \neq \mathbb{G}_m^n$; hence the need to strengthen rotundity of $(z +L) \times W$ to geometrical non-degeneracy of $W$: it is a straightforward calculation that if $W$ is geometrically non-degenerate, then rotundity of $(z+L) \times W$ reduces to $\dim L + \dim W \geq n$, and therefore we have a sort of ``uniform rotundity'' as $z$ varies in $\mathbb{G}_a^n(\mathbb{C})$ and $L$ varies in the Grassmannian of linear subspaces of $\mathbb{G}_a^n$ of codimension $\dim W$. We will exploit this uniformity to prove our results, together with some equidistribution techniques.

\subsection{Strategy of the proofs}

In this subsection we briefly sketch the main steps of the proofs of the main results. Throughout this subsection $L$ will denote a linear subspace of $\mathbb{G}_a^n$ defined over $\mathbb{R}$ and $W$ an irreducible algebraic subvariety of $\mathbb{G}_m^n$ such that $\dim L + \dim W  = n$.

An important ingredient of the proofs is a characterization of  rotundity (Proposition \ref{prop:characterize-rotund}) of $L \times W$, which extends previous results from \cite{Gal23} and \cite{K19}. The main takeaway is that while rotundity is defined as an algebraic notion it has an analytic interpretation: $L \times W$ is rotund if and only if there is a Euclidean open subset $U \subseteq (L \times W)(\mathbb{C})$ such that the map $\delta:U \rightarrow (\ctimes)^n$ defined by $(\ell,w) \mapsto \frac{w}{\exp(\ell)}$ (where division is to be understood componentwise) is open in the Euclidean topology.

If $W \subseteq \mathbb{G}_m^n$ is an algebraic subvariety defined over $\mathbb{C}$, then the information about the behaviour of $W$ as some coordinates of its $\mathbb{C}$-points approach $0$ or $\infty$ is recorded by a semilinar object, the \emph{tropicalization} of $W$, denoted $\trop(W)$. This is a set of polyhedra, where, for the purposes of this introduction, a polyhedron is a subset of $\mathbb{R}^n$ defined by semilinear inequalities. Another characterization of rotundity of $L \times W$ given in Proposition \ref{prop:characterize-rotund} is that $\bigcup_{\tau \in \trop(W)} \tau+L(\mathbb{R})=\mathbb{R}^n$.

If we assume $W$ to be geometrically non-degenerate, the observations above can be strengthened to obtain a uniformity property: Theorem \ref{thm:main-tropical} shows that if $W$ is geometrically non-degenerate, then there is $\epsilon \in \mathbb{R}_{>0}$ such that for all $L'$ of dimension at least $\codim W$ and for all $z \in (\ctimes)^n$ the image of the map $\delta: L'(\mathbb{C}) \times (z \cdot W(\mathbb{C})) \rightarrow (\ctimes)^n$ defined similarly as above contains a ball of radius $\epsilon$ centered in some point of the $n$-th power of the unit circle, $\mathbb{S}_1^n(\mathbb{C})$. This is achieved through a study of the interaction between $L'(\mathbb{R})$ and $\trop(W)$ carried out in Lemma \ref{lem:neighborhood-in-grass} and Proposition \ref{prop:non-archimedean-version}. 

This will allow us to apply Bilu's Equidistribution Theorem \ref{thm:torsion_points_equidistribute} to show that if $L$ is defined over $\mathbb{Q}$, then $\exp(L(\mathbb{C})) \cap z \cdot W(\mathbb{C}) \neq \emptyset$ for all $z \in (\ctimes)^n$ as soon as no linear equation with ``small" integer coefficients is satisfied identically on $L$. The idea is that the integer coefficients of equations satisfied identically on $L$ provide a measure of ``complexity" of $L$ and, as this grows, the set $\exp(L(\mathbb{C})) \cap \mathbb{S}_1^n(\mathbb{C})$ is forced to intersect all balls of fixed radius. 

\subsection{Applications}

We will show in Theorem \ref{thm:main-acf} that Theorem \ref{thm:main} implies an analogous statement over any algebraically closed field $K$ of characteristic $0$. The same is true for Theorem \ref{thm:main-two} as well. We will then derive some consequences.

In Proposition \ref{prop:proof-manin-mumford}, we use Theorem \ref{thm:main-two} to give a new proof of the Manin-Mumford conjecture for algebraic tori, proved by Laurent in \cite{Laurent_1984}, by transforming an unlikely intersection happening into a likely intersection not happening. Given that we use equidistribution, which can also be used to prove the Manin-Mumford conjecture directly, this might not be too surprising. Nevertheless, the transformation of a problem about unlikely intersections into a problem about likely intersections seems not without interest.

If the subvariety $W$ is a curve, then we will show that Theorem \ref{thm:main} can be strengthened as follows.

\begin{coro}\label{cor:curvecase}
Let $K$ be an algebraically closed field of characteristic $0$ and let $W \subseteq \mathbb{G}^n_{m,K}$ be an irreducible curve that is not contained in any coset of a proper subtorus.
	
Then there exist finitely many subtori $H_i \subsetneq \mathbb{G}^n_{m,K}$ ($i = 1,\hdots,N$) and finitely many points $z_1, \hdots, z_N$ in $(K^\times)^n$ such that for every subtorus $H \subseteq \mathbb{G}^n_{m,K}$ with $\dim H = n-1$ and for every $z \in (K^\times)^n$, one of the following holds:
\begin{enumerate}
	\item[(i)] $W(K) \cap z \cdot H(K) \neq \emptyset$ or
	\item[(ii)] $z \cdot H = z_i \cdot H_i$ for some $i \in \{1,\hdots,N\}$.
\end{enumerate}
\end{coro}

We will also apply Theorem \ref{thm:main-acf} to prove that likely intersections are likely in a concrete probabilistic sense. Concretely, we will prove the following theorem.

\begin{thm}\label{thm:probabilistic}
Let $K$ be an algebraically closed field of characteristic $0$ and let $W \subseteq \mathbb{G}^n_{m,K}$ be a geometrically non-degenerate subvariety of dimension $d \in \{1,\hdots,n-1\}$. For $N \in \mathbb{N}$, let $\mathcal{S}(W,N) \subseteq \mathbb{Z}^{nd}$ denote the set of vectors $(a_{i,j})_{i = 1,\hdots,d;~j = 1,\hdots,n}$ such that
\[ \max_{i,j}{|a_{i,j}|} \leq N\]
and such that, for every $y = (y_1,\hdots,y_d) \in (K^\times)^d$, there exists some $x = (x_1,\hdots,x_n) \in W(K)$ with
\[ \prod_{j=1}^{n}{x_j^{a_{i,j}}} = y_i \quad (i = 1, \hdots, d).\]
Then
\[ \lim_{N \to \infty}{\frac{\# \mathcal{S}(W,N)}{(2N+1)^{nd}}} = 1.\]
\end{thm}

While the hypothesis of geometrical non-degeneracy is used in the proof of Theorem \ref{thm:probabilistic}, it might well not be necessary for the theorem to hold. 

\subsection{Structure of the paper} The paper is organized as follows. In Section \ref{sec:preli}, the necessary preliminaries from model theory and tropical geometry are introduced; this allows us to introduce the algebraically closed valued field $\mathfrak{C}$ in which many of our proofs will take place. In Section \ref{sec:rot-and-geo} we discuss Zilber's notion of rotundity and show some equivalent characterizations of rotundity for complex varieties of the form $L \times W$, with $L \leq \mathbb{G}_a^n$ linear and defined over $\mathbb{R}$ and $W \subseteq \mathbb{G}_m^n$ equidimensional of dimension $\codim L$. We then show that the geometrical non-degeneracy of $W$ allows us to strengthen one of these characterizations to a version that is uniform in $L$. In Section \ref{sec:equi} we combine this uniformity result with an equidistribution argument to prove Theorems \ref{thm:main} and \ref{thm:main-two}. Finally, in Section \ref{sec:appli}, we discuss some applications; in particular, we prove Corollary \ref{cor:curvecase} and Theorem \ref{thm:probabilistic}.

\section{Preliminaries}\label{sec:preli}

    \subsection{Model theory} 

We will use very basic model theory in order to obtain some uniformity results, hence we use this subsection to fix some terminology. We will value conciseness over precision, and refer the reader to one of the many excellent textbooks in the area, for example \cite{hodges1997shorter}, for the details.

Recall that a \textit{language} $\mathcal{L}$ consists of symbols for constants, functions, and relations, and that an $\mathcal{L}$-\textit{formula} is an expression involving a symbol $=$ for equality, variables, the symbols in $\mathcal{L}$, and the usual logical symbols $\neg, \vee, \wedge, \exists, \forall$. A \textit{structure} in a language $\mathcal{L}$ is a set $S$ together with an assignment of the constant symbols to some elements of $S$, of the function symbols to some functions $S^n \rightarrow S$, and of the relation symbols to some subsets of $S^m$, for appropriate $n$ and $m$. A formula \textit{with parameters from $A$} for some subset $A \subseteq S$ is a formula $\phi(x,a)$ where $x,y$ are tuples of free variables (i.e., variables not appearing in the scope of a quantifier), $\phi(x,y)$ is an $\mathcal{L}$-formula, and $a$ is a tuple of elements from $A$ (of the same length as $y$).

A subset $X \subseteq S^n$ is \textit{definable} (over $A$) if there is some formula $\phi$ in $n$ free variables (with parameters from $A$) such that $X$ consists of those points of $S^n$ which satisfy $\phi$.

An \textit{elementary extension} of the structure $S$ is an $\mathcal{L}$-structure $S'$ such that $S \subseteq S'$ and every formula with parameters from $S$ which is true in $S$ is also true in $S'$. As an example, the field of complex numbers is not an elementary extension of the field of real numbers: the latter satisfies the sentence ``$\forall x \neg(x^2 = -1)$'' while the former does not. On the other hand, $\mathbb{C}$ is an elementary extension of the algebraic closure $\o{\mathbb{Q}}$ of the rationals, as can be seen for example by applying the \textit{Tarski-Vaught criterion} \autocite[Theorem 2.5.1]{hodges1997shorter} and quantifier elimination for the theory of algebraically closed fields \autocite[Theorem 2.7.3]{hodges1997shorter}.

Given a structure $S$, an elementary extension $S'$ of $S$, and a definable subset $X$ of $S^n$, we may look at the subset of $(S')^n$ defined by the same formula. This set will be denoted by $X(S')$: the geometrically-minded reader will notice the similarity with taking points of algebraic varieties in different fields containing the field of definition. 

    \subsection{Tropical geometry}

    Here we recall the necessary preliminaries from tropical geometry. Most of the material in this subsection is known and follows for example from results in \cite{MS}, but we work out full proofs when we have not been able to find them in the literature. 
    
    Throughout this subsection, we fix an algebraically closed, non-trivially valued field $(K, \val)$, with (ordered) value group $(\Gamma,+,\leq)$, valuation ring $\mathcal{O}$, maximal ideal $\mathfrak{m}$, and residue field $k \coloneqq \mathcal{O}/\mathfrak{m}$. We will assume that the valuation $\val:K^\times \rightarrow \Gamma$ is surjective. By abuse of notation, we also consider $\val$ as a map from $K$ to $\Gamma \sqcup \{\infty\}$, where $\val(0) = \infty$ and $\infty$ is larger than every element of $\Gamma$. We will denote the residue map by $\res: \mathcal{O} \twoheadrightarrow k$. We also write $\res$ for the map on the polynomial rings $\res:\mathcal{O}[x_1^{\pm 1}, \dots, x_n^{\pm 1}]\twoheadrightarrow k[x_1^{\pm 1}, \dots, x_n^{\pm 1}]$ as well as for any Cartesian power $\res: \mathcal{O}^n \twoheadrightarrow k^n$ of $\res$. Note that $\Gamma$ is divisible and $k$ is algebraically closed. In this subsection we will further assume that $\Gamma$ is an Archimedean ordered abelian group, so that it can be embedded in $\mathbb{R}$; in most cases this assumption could be removed without causing any trouble, but we keep it nonetheless since we quote many results from \cite{MS} and it is always assumed there. In our applications, however, $\Gamma$ is \emph{not} going to be Archimedean; at the end of the next subsection we will explain why we do not need to worry about this.

    We fix a \textit{splitting} of the valuation, that is, a group homomorphism $s:\Gamma \rightarrow K^\times$ such that $\val(s(\gamma))=\gamma$ for all $\gamma \in \Gamma$. This exists by \autocite[Lemma 2.1.15]{MS}. With another abuse of notation, we will use $\val$ and $s$ also to denote the Cartesian powers of these maps, that is, the maps $\val:(K^\times)^n \rightarrow \Gamma^n$ and $s:\Gamma^n \rightarrow (K^\times)^n$.

    Given vectors $u \in \mathbb{Q}^n$ and $\gamma \in \Gamma^n$, we write $\langle u, \gamma\rangle$ for the element $u_1\gamma_1+\dots+u_n \gamma_n \in \Gamma$, with the action of $\mathbb{Q}$ on $\Gamma$ defined in the obvious way. This is not to be confused with the operation $v \cdot v'$ on elements $v  = (v_1,\hdots,v_n)$, $v' = (v_1',\hdots,v_n') \in \mathbb{G}_{m,K}^n(K)$, which is given by $v \cdot v' = (v_1 \cdot v_1',\dots,v_n \cdot v_n')$. Finally, if $v \in \mathbb{G}_{m,K}^n(K)$ and $A = (a_{i,j}) \in \mathrm{Mat}_{d \times n}(\mathbb{Z})$, then $v^A=(\Pi_{i=1}^nv_i^{a_{1,i}}, \hdots, \Pi_{i=1}^nv_i^{a_{d,i}})$.

    \begin{defn}
        Let $f \in K[x_1^{\pm 1}, \dots, x_n^{\pm 1}]$ be a Laurent polynomial and let $\gamma \in \Gamma^n$. We write $f$ as $$f=\sum_{u \in S} c_u x^u,$$ where $x = (x_1, \hdots, x_n)$ and $S \subseteq \mathbb{Z}^n$ is a finite subset (we assume $c_u \neq 0$ for all $u \in S$). 

        The \textit{initial form} of $f$ with respect to $\gamma$ is $\ini_\gamma(f) \in k[x_1^{\pm 1}, \dots, x_n^{\pm 1}]$ defined as $$\ini_\gamma(f)=\sum_{u \in S'} \res(c_us(-\val(c_u)))x^u$$ where $$S':=\{u \in S \mid \val(c_u)+\langle u,\gamma \rangle \leq\val(c_{u'})+ \langle u',\gamma \rangle \, \forall u' \in S  \}.$$
    \end{defn}

    For example, if $\val(c_u)\geq 0$ for each $u \in S$ and $\val(c_{u'})$ is equal to $0$ for some $u' \in S$, then $\ini_0(f)= \res(f)$.

    Given an ideal $I \subseteq K[x_1^{\pm 1}, \hdots, x_n^{\pm 1}]$ and $\gamma \in \Gamma^n$, we define the \textit{initial ideal} $$\ini_\gamma(I) \coloneqq \langle \ini_\gamma(f) \mid f \in I \rangle \subseteq k[x_1^{\pm 1}, \hdots, x_n^{\pm 1}].$$
    
    \begin{lem}[{\autocite[Lemma 2.6.2(1)]{MS} and \cite{MSErrata}}]\label{lem:initial-is-initial}
        If $g \in \ini_\gamma(I)$, then $g=\ini_\gamma(f)$ for some $f \in I$.
    \end{lem}

    \begin{defn}
        Let $W$ be an algebraic subvariety of $\mathbb{G}_{m,K}^n$, and let $I$ be the ideal of Laurent polynomials which vanish on $W$.

        For $\gamma \in \Gamma^n$, the \textit{initial variety} of $W$ with respect to $\gamma$ is the algebraic subvariety of $\mathbb{G}_{m,k}^n$ defined by $\ini_\gamma(I)$. It is denoted by $\ini_\gamma(W)$.
    \end{defn}

    \begin{lem}\label{lem:componentsofinitialform}
        Let $W$ be an algebraic subvariety of $\mathbb{G}_{m,K}^n$ with irreducible components $W_1, \hdots, W_\ell$ and let $\gamma \in \Gamma^n$. Then
            \[ \ini_\gamma(W) = \bigcup_{j=1}^{\ell} \ini_\gamma(W_j).\]
        If $W$ is irreducible, then every irreducible component of $\ini_\gamma(W)$ has dimension $\dim W$.
    \end{lem}

    \begin{proof}
        For $j = 1, \hdots, \ell$, let $I_j$ denote the prime ideal of $K[x_1^{\pm 1},\hdots,x_n^{\pm 1}]$ associated to $W_j$. Let $I$ denote the (radical) ideal associated to $W$. Then
        \[ \prod_{j=1}^{\ell}{I_j} \subseteq I \subseteq \bigcap_{j=1}^{\ell}{I_j}.\]
        It follows that $\ini_\gamma(I) \subseteq \bigcap_{j=1}^{\ell}{\ini_\gamma(I_j)}$ and so
        \begin{equation}\label{eq:initialformscontainment}
        \bigcup_{j=1}^{\ell}{\ini_\gamma(W_j)} \subseteq \ini_\gamma(W).
        \end{equation}
        
        Moreover, if $g_j \in \ini_\gamma(I_j)$ are arbitrary ($j = 1,\hdots,\ell$), then Lemma \ref{lem:initial-is-initial} implies that each $g_j$ is equal to $\ini_\gamma(f_j)$ for some $f_j \in I_j$ and so
        \[ \prod_{j=1}^{\ell}{g_j} = \ini_\gamma\left(\prod_{j=1}^{\ell}f_j\right) \in \ini_\gamma(I_1 \cdots I_\ell) \subseteq \ini_\gamma(I),\]
        where we have used \cite[Lemma 2.6.2(3)]{MS}. Hence, $\prod_{j=1}^{\ell}{\ini_\gamma(I_j)} \subseteq \ini_\gamma(I)$. Now, if 
        \[\bigcup_{j=1}^{\ell}{\ini_\gamma(W_j)} \neq \ini_\gamma(W),\]
         then, by \eqref{eq:initialformscontainment}, there would exist $x \in (\ini_\gamma(W))(K)$ and Laurent polynomials $g_j \in \ini_\gamma(I_j)$ such that $g_j(x) \neq 0$ for all $j$. But then the product of the $g_j$'s could not belong to $\ini_\gamma(I)$, a contradiction.

        The dimension statement is \cite[Lemma 3.2.6]{MS}.
    \end{proof}

    \begin{prop}\label{prop:hensel}
        Let $W \subseteq \mathbb{G}_{m,K}^n$ be an algebraic subvariety, $\gamma \in \Gamma^n$. If $v \in \ini_\gamma(W)(k)$, then there is $v' \in W(K)$ such that $\val(v')=\gamma$ and $\res(v' \cdot s(-\gamma))=v$.
    \end{prop}

    \begin{proof}
        If $W$ is irreducible, this is \cite[Proposition 3.2.11]{MS} (whose proof uses irreducibility to draw a stronger conclusion and actually already yields our proposition for reducible subvarieties). The general case then follows from Lemma \ref{lem:componentsofinitialform}.
    \end{proof}

    \begin{thm}[{\autocite[Theorem 3.2.3]{MS}}]\label{thm:trop-is-val-of-points}
        Let $W \subseteq \mathbb{G}_{m,K}^n$ be an algebraic subvariety and $I$ the ideal of Laurent polynomials which vanish on $W$. Then the following sets coincide:
        \begin{itemize}
            \item[$(1)$] $\{\val(v) \in \Gamma^n \mid v \in W(K) \}$;
            \item[$(2)$] $\{\gamma \in \Gamma^n \mid \ini_{\gamma}(I)\neq \langle 1 \rangle\}.$ 
        \end{itemize}
    \end{thm}

    We remark that in \cite[Theorem 3.2.3]{MS} the closure of the image of $W(K)$ under $\val$ is considered. For us, this is not necessary, as one can see using Proposition \ref{prop:hensel} that set (2) is contained in set (1) and not just in its closure (for the other inclusion, see \cite[proof of Theorem 3.2.3]{MS}).
    
    A \textit{($\Gamma$-rational) polyhedron} is a subset $\tau$ of $\Gamma^n$ consisting of the elements $\gamma$ which satisfy finitely many inequalities of the form $\langle u, \gamma \rangle \leq \gamma_0$, for some $u \in \mathbb{Q}^n$ and $\gamma_0 \in \Gamma$. A \textit{face} of the polyhedron $\tau$ is a polyhedron $\tau'$ for which there exists some $u \in \mathbb{Q}^n$ such that $\tau'=\{\gamma \in \tau \mid \langle u, \gamma \rangle \leq \langle u, \gamma' \rangle \, \forall \gamma' \in \tau \}.$ The \textit{relative interior} of the polyhedron $\tau$, denoted $\relint(\tau)$, is the set of points in $\tau$ which do not lie in any proper face of $\tau$.
    
     A \textit{($\Gamma$-rational) polyhedral complex} is a finite set $\Sigma$ of non-empty $\Gamma$-rational polyhedra that is closed under taking non-empty faces, and such that for all $\tau_1,\tau_2 \in \Sigma$ we have that, if $\tau_1 \cap \tau_2$ is non-empty, then it is a face of both $\tau_1$ and $\tau_2$ (and hence an element of $\Sigma$). The \textit{support} $\supp(\Sigma)$ of the polyhedral complex $\Sigma$ is the union $\bigcup_{\tau \in \Sigma} \tau \subseteq \Gamma^n$.

    The \textit{affine span} of a polyhedron $\tau$, denoted $\aff(\tau)$, is the smallest affine subspace of $\Gamma^n$ defined by equations of the form $\langle u, \gamma \rangle = \gamma_0$ ($u \in \mathbb{Q}^n$, $\gamma_0 \in \Gamma$) and containing $\tau$. We denote by $\lin(\tau)$ the linear subspace of $\Gamma^n$ that is parallel to $\aff(\tau)$. There is a matrix $A$ with rational entries such that $\lin(\tau)=\{\gamma \in \Gamma^n \mid A \cdot \gamma=0\}.$ We define $\lin(\tau)(\mathbb{R})$ and $\lin(\tau)(\mathbb{C})$ as the kernels of $A$ in $\mathbb{R}^n$ and $\mathbb{C}^n$ respectively. The \emph{dimension} $\dim(\tau)$ of $\tau$ is the dimension of $\lin(\tau)(\mathbb{R})$. The \emph{dimension} of a polyhedral complex is the maximal dimension of one of its polyhedra.
        
    \begin{thm}[{\autocite[Theorem 3.2.3, Proposition 3.2.8, and Theorem 3.3.8]{MS}}]\label{thm:structure-theorem}
        Let $W \subseteq \mathbb{G}_{m,K}^n$ be an algebraic subvariety. The set $\val(W(K))$ is the support of a polyhedral complex of dimension $\dim W$.
    \end{thm}

    It follows from the existence of universal Gr\"obner bases \autocite[Corollary 2.5.11]{MS} together with \cite[Proposition 2.6.1]{MS} that the polyhedral complex can be chosen so that for all  $\gamma_1,\gamma_2 \in \val(W(K))$, if $\gamma_1,\gamma_2$ lie in the relative interior of the same polyhedron $\tau$ in the complex, then $\ini_{\gamma_1}(I)=\ini_{\gamma_2}(I)$. If we fix such a choice, we are allowed to write $W_\tau$ instead of $\ini_\gamma(W)$ and $\ini_\tau(I)$ instead of $\ini_\gamma(I)$ whenever $\gamma \in \relint(\tau)$.

    \begin{defn}
    Let $W \subseteq \mathbb{G}_{m,K}^n$ be an algebraic subvariety. A \textit{tropicalization} of $W$ is a polyhedral complex with support $\val(W(K))$ such that for every polyhedron $\tau$ in the complex and all $\gamma_1,\gamma_2 \in \relint(\tau)$, we have $\ini_{\gamma_1}(W)=\ini_{\gamma_2}(W)$.
    \end{defn}

    While the polyhedral complex above is not determined uniquely (see \cite[Example 3.5.4]{MS}), it always exists. We abuse notation and write $\trop(W)$ to denote one such complex.

    Hence, we have $\val(W(K))=\supp(\trop(W))$.

    \begin{lem}\label{lem:equal-residue}
         Let $W \subseteq \mathbb{G}_{m,K}^n$ be an algebraic subvariety, $I$ the ideal of Laurent polynomials vanishing on $W$. Let $v \in W(K)$, $\gamma=\val(v) \in \Gamma^n$, $W_\tau=\ini_{\gamma}(W)$, and $a=s(\gamma)$. 
         
        Then 
        \[\ini_\gamma(I) = \{\res(g) \mid g \in \mathcal{O} [x_1^{\pm 1},\hdots,x_n^{\pm 1}],~g|_{a^{-1} \cdot W} = 0\}\]
        and 
        \begin{equation}\label{eq:ini-var-equality}
            W_\tau(k)=\res((a^{-1} \cdot W(K)) \cap (\mathcal{O}^\times)^n).
        \end{equation}
    \end{lem}

    \begin{proof}
    We first prove the equality of the two sets of polynomials. 
    
    $(\supseteq)$ Let $g$ be a Laurent polynomial over $\mathcal{O}$ which vanishes on $a^{-1} \cdot W$; then $g(a^{-1} \cdot x) \in I$. If $\res(g)=0$ then trivially $\res(g) \in \ini_\gamma(I)$, so we assume that $\res(g) \neq 0$, that is, $g$ has at least one coefficient of valuation 0.
    
    Write $g=\sum_{u \in S} c_u x^u$. Then $g(a^{-1} \cdot x)=\sum_{u \in S}  c_ua^{-u} x^u $. Let 
    \[S':=\{u \in S \mid \val(c_ua^{-u})+\langle u,\gamma \rangle  \leq \val(c_{u'} a^{-u'}) + \langle u', \gamma \rangle  \, \forall u' \in S\}. \]
    Since $\val(a)=\gamma$, we have that $u \in S'$ if and only if $\val(c_u) \leq \val(c_{u'})$ for all $u' \in S$. By assumption on $g$, this happens if and only if $\val(c_u)=0$.
    Hence 
    \begin{align*}
        \ini_\gamma(g(a^{-1} \cdot x)) &= \sum_{u \in S'} \res \left( c_ua^{-u} s(-\val(c_u a^{-u})) x^u \right) \\
        &=\sum_{u \in S'} \res \left( c_u s(-\val(c_u)) x^u \right) \\
        &=\res \left( \sum_{u \in S} c_ux^u \right) \\
        &=\res(g)
    \end{align*}
    and thus $\res(g) \in \ini_\gamma(I)$.
    
    $(\subseteq)$ By Lemma \ref{lem:initial-is-initial}, it is sufficient to show that, for every $f\in I$, $\ini_\gamma(f)$ belongs to the right-hand set. So let $f \in I$. If $f=0$, there is nothing to prove, so we assume $f \neq 0$, and write $f=\sum_{u \in S} c_ux^u$. Let $$S':=\{u \in S \mid \val(c_u)+ \langle u, \gamma \rangle \leq \val(c_{u'})+\langle u', \gamma \rangle \, \forall u' \in S\}.$$ Take some $u_0 \in S'$, and let $\lambda:=c_{u_0} a^{u_0}$.
    
     Then, for all $u \in S$, 
        \begin{align*}
            \val\left( \frac{c_u a^u}{\lambda} \right) 
            &=\val(c_u)+\val\left( a^u \right) - \val(\lambda) \\
            &=\val(c_u)+\langle u, \gamma \rangle - \val(\lambda) \\
            &=\val(c_u)+\langle u, \gamma \rangle - \val(c_{u_0})- \langle u_0, \gamma \rangle \\
            & \geq 0
        \end{align*}
        with equality if and only if $u \in S'$, and hence the coefficients of the polynomial 
    \begin{align*}
        \frac{f(a \cdot x)}{s(\val(\lambda))} &=\frac{\sum_{u \in S} c_u (a \cdot x)^u}{s(\val(\lambda))}  \\
        &=\sum_{u \in S} \frac{c_ua^u}{s(\val(\lambda))} x^u
    \end{align*}
    have non-negative valuation, strictly positive for all $u \in S \setminus S'$. Then
    \begin{align*}
        \res \left( \frac{f(a \cdot x)}{s(\val(\lambda))} \right) &= \res \left( \frac{\sum_{u \in S} c_u (a \cdot x)^u}{s(\val(\lambda))} \right) \\
        &=\res \left( \frac{\sum_{u \in S'} c_u(a \cdot x)^u}{s(\val(\lambda))} \right) + \res \left( \frac{\sum_{u \in S \setminus S'} c_u(a \cdot x)^u}{s(\val(\lambda))} \right) \\
        &=\res \left(\sum_{u \in S'} \frac{c_u (a \cdot x)^u} {s(\val(c_u)+\langle u,\gamma \rangle) }\right) +0 \\
        &=\res \left( \sum_{u \in S'} c_us(-\val(c_u)) x^u \right)\\
        &=\ini_\gamma \left( f \right).
    \end{align*}

    Hence if $f \in I \backslash \{0\}$, then $\ini_\gamma(f)=\res \left( \frac{f(a \cdot x)}{s(\val(\lambda))}\right)$, and $\frac{f(a \cdot x)}{s(\val(\lambda))}$ is a polynomial with coefficients in $\mathcal{O}$ which vanishes on $a^{-1} \cdot W$. This completes the proof of the first equality.

    We now prove the second equality.

        $(\subseteq)$ Let $v_0 \in W_\tau(k)$. Then by Proposition \ref{prop:hensel}, there is $v_1 \in W(K)$ such that $\val(v_1)=\gamma$ and $\res(v_1 \cdot s(-\gamma))=\res(a^{-1} \cdot v_1)=v_0$, so $v_0 \in \res \left( (a^{-1} \cdot W(K)) \cap (\mathcal{O}^\times)^n \right)$.

        $(\supseteq)$ Let $v_1 \in  (a^{-1} \cdot W(K)) \cap (\mathcal{O}^\times)^n$, $f \in \ini_\gamma(I)$. By the first part of the proof, there is $g \in \mathcal{O}[x_1^{\pm 1},\dots,x_n^{\pm 1}]$ such that $g$ vanishes on $a^{-1} \cdot W$ (so in particular $g(v_1)=0$) and $\res(g)=f$. Then, $f(\res(v_1))=\res(g)(\res(v_1))=\res(g(v_1))=0$. Thus $\res(v_1) \in W_\tau(k)$.
    \end{proof}

    \begin{defn}\label{def:jtau}
        Let $\tau \leq \Gamma^n$ be a polyhedron, $A \in \mathrm{Mat}_{d \times n}(\mathbb{Z})$ a matrix with integer entries such that $\aff(\tau)=\{x \in \Gamma^n \mid Ax=b \}$ for some $b \in \Gamma^d$.

        We denote by $J_{\tau}$ the identity component of the algebraic subgroup of $\mathbb{G}_{m,k}^n$ defined by $y^A=1$.
    \end{defn}

    In the next proof, we will take initial forms of initial forms. To be able to do this, we consider an algebraic closure $\overline{k((\Gamma))}$ of the Hahn field $k((\Gamma))$ (for a general treatment of Hahn fields, see for example \cite[Chapter 3]{Serra2021Autom}). Hence, for $f =\sum_{u \in S} c_ux^u\in k[x_1^{\pm 1}, \dots, x_n^{\pm 1}]$ and $\gamma \in \Gamma^n$ we have that $\ini_\gamma(f):=\sum_{u \in S'} c_ux^u$ where $S':=\{u \in S \mid \langle u, \gamma \rangle \leq \langle u', \gamma \rangle \, \forall u' \in S \}$.
    
    \begin{prop}\label{prop:invariant}
        Let $W \subseteq \mathbb{G}_{m,K}^n$ be an algebraic subvariety and $\tau \in \trop(W)$ a polyhedron. Then $J_\tau \leq \stab(W_{\tau,\mathrm{Irr}})$ for every irreducible component $W_{\tau,\mathrm{Irr}}$ of $W_\tau$, in particular $J_\tau \leq \stab(W_\tau)$.
    \end{prop}

    \begin{proof}
        We first show that $J_\tau \leq \stab(W_\tau)$.
        
        Let $I$ be the ideal of Laurent polynomials vanishing on $W$ and let $\gamma \in \relint(\tau)$.

        If $\dim J_{\tau} = 0$, then $J_\tau \leq \stab(W_\tau)$ holds trivially.

        Hence, we can assume that $\dim J_\tau > 0$. Then the set of points in $J_\tau(k)$ which lie in an irreducible 1-dimensional algebraic subgroup of $J_\tau$ is Zariski dense in $J_\tau$; hence, it is sufficient to check that $W_\tau$ is stabilized by these points, as the stabilizer is a (Zariski closed) algebraic subgroup of $\mathbb{G}_{m,k}^n$.

        Let $H$ be an irreducible 1-dimensional algebraic subgroup of $J_{\tau}$. We have
        $H(k) = \{(t^{u_1}, \hdots, t^{u_n}) \mid t \in k^\times\}$
        for some $u = (u_1,\hdots,u_n) \in \mathbb{Z}^n \setminus \{0\}$.
        Assume then that $z \in H(k)$ so that $z = (t_0^{u_1}, \hdots, t_0^{u_n})$ for some $t_0 \in k^\times$.
        Set $v = (u_1 \gamma_0,\hdots,u_n \gamma_0)$ for some $\gamma_0 \in \Gamma \setminus \{0\}$. Since $H \subseteq J_\tau$, we have that $v \in \lin(\tau)$; since $\gamma \in \relint(\tau)$, we may then choose $\gamma_0$ so that $\gamma+v \in \relint(\tau)$.

        The proofs of \cite[Lemma 2.4.6 and Corollary 2.4.10]{MS} can be easily adapted to non-homogeneous ideals to show that, possibly after replacing $v$ by some rational multiple of itself, $\ini_v(\ini_\gamma(I)) \subseteq \ini_{\gamma + v}(I)$. It then follows from \cite[Corollary 2.4.10 and Proposition 2.6.1]{MS} and the fact that dehomogenizing and then rehomogenizing can only make an ideal larger that also $\ini_{\gamma + v}(I) \subseteq \ini_v(\ini_\gamma(I))$ (after possibly replacing $v$ again). Hence we have $\ini_\tau(I) = \ini_{\gamma+v}(I)=\ini_v(\ini_\gamma(I)) = \ini_v(\ini_\tau(I))$.

        We apply \cite[Lemma 2.6.2(2)]{MS} to deduce that $\ini_\tau(I)$ is homogeneous with respect to the grading that weights the $i$-th variable by the $i$-th entry $v_i$ of $v$ ($i = 1,\hdots,n$). Let $w \in W_\tau(k)$, and fix homogeneous generators $f_1,\dots,f_\ell$ for $\ini_\tau(I)$. We deduce that for each $i$ we have:
        \[ f_i(z \cdot w) = z_if_i(w)=0\] for some $z_i \in k^\times$. It follows that $z \cdot w \in W_\tau(k)$.
        Since $w \in W_\tau(k)$ was arbitrary, it follows that $z \in \stab(W_\tau)(k)$. Since $z$ and $H$ were arbitrary, we deduce that $J_\tau \leq \stab(W_\tau)$.

        Finally, if $N$ denotes the number of irreducible components of $W_\tau$, then, for every $g \in J_\tau(k) \leq \stab(W_\tau)(k)$, we have that $g^{N!}$ stabilizes each irreducible component of $W_\tau$ individually. But $J_\tau$ is connected, so $J_\tau(k)$ is divisible and we are done.
    \end{proof}

    Recall that a \emph{cone} is a polyhedron which is defined by inequalities of the form $\langle u, \gamma \rangle \leq 0$, and that a \emph{fan} is a polyhedral complex whose polyhedra are all cones.

    \begin{defn}
    Let $\Sigma_1, \Sigma_2$ be polyhedral complexes in $\Gamma^n$ and $\gamma$ a point in the intersection of their supports. We say $\Sigma_1$ and $\Sigma_2$ \emph{locally coincide around $\gamma$} if there is an open box $B \subseteq \Gamma^n$ containing $\gamma$ such that for every $\tau \in \Sigma_1$ (resp$.$ $\Sigma_2$) with $\gamma \in \tau$, there is $\tau' \in \Sigma_{2}$ (resp$.$ $\Sigma_1$) such that $\tau \cap B=\tau' \cap B$.
    \end{defn}

    \begin{defn}
        Let $\Sigma$ be a polyhedral complex in $\Gamma^n$, $\tau \in \Sigma$, $\gamma \in \relint(\tau)$. The \emph{star} of $\Sigma$ around $\tau$, denoted $\star_{\Sigma}(\tau)$, is the uniquely determined fan which locally coincides with $\Sigma-\gamma$ around $0$. The star does not depend on the choice of $\gamma \in \relint(\tau)$.
    \end{defn}

    One can check that this definition coincides with \cite[Definition 2.3.6]{MS} (for $\Gamma = \mathbb{R}$). The  fact that $\star_{\Sigma}(\tau)$ and $\Sigma - \gamma$ locally coincide around $0$ induces a bijective correspondence between the star and the set of polyhedra in $\Sigma$ that contain $\tau$; this correspondence is also independent of the choice of $\gamma$.

    \begin{prop}[{\autocite[Lemma 3.3.6]{MS}}]\label{prop:star-is-tropical}
        Let $W \subseteq \mathbb{G}^n_{m,K}$ be an algebraic subvariety and let $\tau \in \trop(W)$ be a polyhedron and $W_\tau \subseteq \mathbb{G}^n_{m,\overline{k((\Gamma))}}$ the corresponding initial variety. 

        Then $\star_{\trop(W)}(\tau)$ is a tropicalization of $W_\tau$.
    \end{prop}

    The proof of \cite[Lemma 3.3.6]{MS} shows that $\star_{\trop(W)}(\tau)$ has the property that, for $\gamma_1, \gamma_2$ in the relative interior of the same cone, the initial varieties $\ini_{\gamma_1}(W_\tau)$ and $\ini_{\gamma_2}(W_\tau)$ are the same.

    \begin{rem}\label{rem:invariant-star}
        Every cone $\sigma \in \star_{\trop(W)}(\tau)$ contains $\lin(\tau)$, and thus $\trop(W_\tau)$ is invariant under translation by $\lin(\tau)$. This is the tropical counterpart to Proposition \ref{prop:invariant}.
    \end{rem}

    If $W \subseteq \mathbb{G}_m^n$ is an algebraic subvariety, we may see it as a subvariety of $\mathbb{G}^n_{m,K}$ for some algebraically closed non-trivially valued field extension $K$ of $\mathbb{C}$, for example the Puiseux series field, with Archimedean value group and with the trivial valuation induced on $\mathbb{C}$. It is natural to expect that in this case the tropicalization should have a simple structure. This is true: we will explain this in Remark \ref{rem:constant-coefficients} and see a consequence, related to Lemma \ref{lem:equal-residue}, in Remark \ref{rem:equal-residue}. However we first need a technical lemma saying that being a universal Gr\"obner basis for a homogeneous ideal is preserved under extensions of algebraically closed valued fields with Archimedean value group. In particular, we need to allow ourselves to take initial forms of homogeneous polynomials and ideals over (possibly) trivially valued fields, see \cite[p.~66]{MS}. For this, we fix an embedding of the value group as an ordered subgroup of $\mathbb{R}$. This is slightly different from the setting that we have been working in and will only be relevant in Lemma \ref{lem:preserve-gr-basis} and Remark \ref{rem:constant-coefficients}. See \cite[Sections 2.4 and 2.5]{MS} for more details on universal Gr\"obner bases.

    \begin{lem}\label{lem:preserve-gr-basis}
    Let $F$ be an algebraically closed subfield of $K$, equipped with the (possibly trivial) induced valuation. Let $I \subseteq F[x_0,\hdots,x_n]$ be a homogeneous ideal and suppose that $\{f_1,\hdots,f_N\}$ is a universal Gr\"obner basis of $I$. Then $\{f_1,\hdots,f_N\}$ is also a universal Gr\"obner basis of the ideal generated by $I$ in $K[x_0,\hdots,x_n]$.
    \end{lem}

    \begin{proof}
    Let $J$ denote the ideal generated by $I$ in $K[x_0,\hdots,x_n]$. Let $\gamma \in \Gamma^{n+1}$. We want to show that $\ini_\gamma(f_1), \hdots, \ini_\gamma(f_N)$ generate $\ini_\gamma(J)$ in $k[x_0,\hdots,x_n]$. Let $\gamma' \in \Gamma^{n+1}$ arbitrary. It follows from \cite[Lemma 2.4.6 and Corollary 2.4.10]{MS} that for a sufficiently small $\epsilon \in \mathbb{Q}_{>0}$, we have
    \[ \ini_{\gamma+\epsilon \gamma'}(J) = \ini_{\gamma'}(\ini_\gamma(J))\]
    and
    \[ \ini_{\gamma+\epsilon \gamma'}(f_i) = \ini_{\gamma'}(\ini_\gamma(f_i))\]
    for $i = 1, \hdots, N$.
    If we can show that the ideal $\ini_{\gamma+\epsilon \gamma'}(J)$ is generated by $\ini_{\gamma+\epsilon \gamma'}(f_1), \hdots, \ini_{\gamma+\epsilon \gamma'}(f_N)$, then we have that $\ini_\gamma(f_1),\dots,\ini_\gamma(f_N)$ is a Gr\"obner basis for $\ini_\gamma(J)$ with respect to $\gamma'$. We can then conclude, from the fact that any Gr\"obner basis generates its ideal in the homogeneous case (\cite[Remark 2.4.4]{MS}), that also $\ini_\gamma(J)$ is generated by $\ini_\gamma(f_1), \hdots, \ini_\gamma(f_N)$.
    
    We are free in how to choose $\gamma'$ and so we can take $\gamma'$ such that the values $\langle u, \gamma' \rangle \in \Gamma$ are pairwise distinct for all $u \in \mathbb{Z}^{n+1}$ such that $x^u$ appears in some $f_i$ with a non-zero coefficient. It follows that, for $\epsilon$ small enough, in all initial forms $\ini_{\gamma+\epsilon \gamma'}(f_i)$, only one monomial occurs with a non-zero coefficient and so $\ini_{\gamma + \epsilon \gamma'}(I)$ is generated by monomials. From now on, we replace $\gamma$ with some such $\gamma + \epsilon \gamma'$, which also lies in $\Gamma^{n+1}$ since $\Gamma$ is divisible, and assume that $\ini_\gamma(I)$ is generated by monomials. We want to show that $\ini_\gamma(J)$ is generated by $\ini_\gamma(f_1), \hdots, \ini_\gamma(f_N)$, which, since $\{f_1,\hdots,f_N\}$ is a universal Gr\"obner basis for $I$, is equivalent to $\ini_\gamma(J)$ being generated by $\ini_\gamma(I)$ in $k[x_0,\hdots,x_n]$.

    If not, then there exists $g \in \ini_\gamma(J)$ such that $g$ is homogeneous and every monomial that occurs in $g$ with a non-zero coefficient does not lie in the ideal generated by $\ini_\gamma(I)$ in $k[x_0,\hdots,x_n]$. By \cite[Lemma 2.4.2]{MS} and \cite{MSErrata}, there exists $f \in J$ such that $g = \ini_\gamma(f)$. We can and will assume that $f$ is homogeneous of degree $\deg g$.

    Let $U$ denote the set of exponents of monomials of degree $\deg f$ in $\ini_\gamma(I)$. It follows from \cite[proof of Lemma 2.4.8]{MS} that there is a basis $\{h_u\}_{u \in U}$ of the homogeneous part of $I$ of degree $\deg f$ such that $\ini_\gamma(h_u)=x^u$ for each $u \in U$. After rescaling, we will assume that the coefficient of $x^u$ in $h_u$ is $1$. Hence, $f = \sum_{u \in U}{\lambda_u h_u}$ for some $\lambda_u \in K$ ($u \in U$). For $u' \in \mathbb{Z}^{n+1}$, we denote by $c_{uu'}$ the coefficient of $x^{u'}$ in $h_u$. Since $\ini_\gamma(h_u) = x^u$, we have that
    \[ \val(c_{uu'}) + \langle u', \gamma \rangle >\val(1)+\langle u, \gamma \rangle = \langle u, \gamma \rangle\]
    for all $u \in U$ and all $u' \neq u$.

    But now
    \begin{align*}
    \min_{u' \in \mathbb{Z}^{n+1}}\left\{\val\left(\sum_{u \in U}{\lambda_uc_{uu'}}\right) + \langle u', \gamma \rangle\right\} &\geq \min_{u,u'}\left\{\val(\lambda_u) + \val(c_{uu'}) + \langle u', \gamma \rangle\right\} \\
    &=\min_{u \in U}\left\{\val(\lambda_u) + \langle u, \gamma \rangle\right\}.
    \end{align*}
    If $u_0 \in U$ is such that $\val(\lambda_{u_0}) + \langle u_0, \gamma \rangle$ is minimal, then
    \begin{align*} \val(\lambda_uc_{uu_0}) + \langle u_0, \gamma \rangle &= \val(\lambda_u) + \val(c_{uu_0}) + \langle u_0, \gamma \rangle\\ &> \val(\lambda_{u}) + \langle u, \gamma \rangle\\ &\geq \val(\lambda_{u_0}) + \langle u_0, \gamma \rangle
    \end{align*}
    for every $u \in U \setminus \{u_0\}$ and so
    \[ \val\left(\sum_{u \in U}{\lambda_u c_{uu_0}}\right) + \langle u_0, \gamma \rangle = \val(\lambda_{u_0}) + \langle u_0, \gamma \rangle .\]
    It follows that $x^{u_0}$ appears in $\ini_\gamma(f) = g$ with a non-zero coefficient, a contradiction.
    \end{proof}

    \begin{rem}\label{rem:constant-coefficients}
    Let now $W \subseteq \mathbb{G}_m^n$ be an algebraic subvariety, which we see as a subvariety of $\mathbb{G}^n_{m,K}$ for some algebraically closed non-trivially valued field extension $K$ of $\mathbb{C}$, with Archimedean value group and with the trivial valuation induced on $\mathbb{C}$. In tropical geometry this is known as the \emph{constant coefficients case}, see \cite[Remark 2.4.5]{MS} for more details.

    Using a universal Gr\"obner basis \cite[Corollary 2.5.11]{MS} and \cite[Proposition 2.6.1]{MS}, we can choose $\trop(W)$ such that each $\tau \in \trop(W)$ is defined by conditions of the form
    \[\val(c_u)+\langle u,\gamma \rangle \leq \val(c_{u'})+\langle u',\gamma \rangle,\]
    which in this case then take the form $\langle u,\gamma \rangle \leq \langle u',\gamma \rangle$ because the universal Gr\"obner basis can be chosen such that all coefficients of all its elements are complex numbers and thus have valuation 0 (being a universal Gr\"obner basis is preserved with respect to the extension $K/\mathbb{C}$ by Lemma \ref{lem:preserve-gr-basis}).
    
    In particular the polyhedral complex $\trop(W)$ can then be chosen to be a fan and so, for each $\tau \in \trop(W)$, we may look at the set $\tau(\mathbb{R})$ of points in $\mathbb{R}^n$ satisfying the given semi-linear inequalities, and similarly for Boolean combinations of polyhedra. Note that, if two finite sets of semi-linear inequalities define the same polyhedron in $\Gamma^n$, then they also do in $\mathbb{R}^n$ by completeness of the theory of non-trivial ordered $\mathbb{Q}$-vector spaces (see \cite[Remark 7.9 in Chapter 1]{Dri98}).
    \end{rem}

    \begin{rem}\label{rem:equal-residue}
    With the same assumptions as in Remark \ref{rem:constant-coefficients}, we now suppose furthermore that the restriction of the residue map to $\mathbb{C}$ induces an isomorphism between $\mathbb{C}$ and the residue field, which yields an embedding of $k$ into $K$.
    
    Returning to Lemma \ref{lem:equal-residue}, if $\tau \in \trop(W)$ is chosen such that $\gamma \in \relint(\tau)$, then we automatically have that $a = s(\gamma) \in J_\tau(K)$ since $s$ is a group homomorphism, $\tau$ is a cone, and it follows from the definition of $J_\tau$ that $s(\lin(\tau)) \subseteq J_\tau(K)$. We now deduce from Proposition \ref{prop:invariant} that
        \[W_\tau(k)=\res((a^{-1} \cdot W(K)) \cap (\mathcal{O}^\times)^n)\]
        for any $a \in J_\tau(K)$ (not necessarily equal to $s(\gamma)$) such that $\val(a) = \gamma$. Note that this statement does not involve the splitting $s$ anymore, which will be convenient for us later.
    \end{rem}

     Finally, we recall the definition of \textit{amoeba}. We will use later that tropicalizations of complex varieties are limits of amoebas.

    \begin{defn}
        Let $\Log:(\ctimes)^n \rightarrow \mathbb{R}^n$ denote the map $$(z_1,\dots,z_n) \mapsto (\log|z_1|,\dots,\log|z_n|).$$ 

        Given an algebraic subvariety $W \subseteq \mathbb{G}_m^n$, the \textit{amoeba} $\mathcal{A}_W$ of $W$ is the image of $W(\mathbb{C})$ under $-\Log$.
    \end{defn}
    
     \subsection{Embedding $\mathbb{C}$ in a large valued field}

    We use an approach based on non-standard analysis: we embed $\mathbb{C}$ into a larger field, naturally endowed with a valuation. This field will contain infinite and infinitesimal elements, which one can intuitively think of as sequences in $\mathbb{C}$ going to $\infty$ or $0$; most of the statements and proofs in the next section could be formulated in the standard language, but that would make them far less readable. 
    
    Consider the model-theoretic structure $\mathbb{R}_{\exp,\sin}$, that is the field of real numbers in the language of ordered rings expanded by constant symbols for all the elements of $\mathbb{R}$ and function symbols for $\exp$ and $\sin$ (here by $\sin$ we mean the \textbf{total} sine function, not the restricted version, so this structure is not o-minimal, and in fact not tame in any sense). We may identify $\mathbb{C}$ with $\mathbb{R}^2$ in the usual way, and using $\exp$ and $\sin$ we may define the complex exponential function in this structure.

    We denote by $\mathfrak{R}$ a proper elementary extension of $\mathbb{R}_{\exp,\sin}$, and by $\mathfrak{C}=\mathfrak{R}+i\mathfrak{R}$ its algebraic closure. We have a canonical map $\mathrm{Re}: \mathfrak{C} \to \mathfrak{R}$, sending $x+iy$ to $x$.

    We may define a map $|\cdot|:\mathfrak{C} \rightarrow \mathfrak{R}$ by setting $|x+iy|=\sqrt{x^2+y^2}$, which, by abuse of language, we call an absolute value  even though its codomain is strictly larger than $\mathbb{R}_{\geq 0}$. We use this to define the \textit{Archimedean valuation} on $\mathfrak{C}$ as the quotient of $\mathfrak{C}$ by the relation $z_1 \sim z_2$ if and only if there is some non-zero $n \in \mathbb{N}$ such that $n|z_1| \geq |z_2|$ and $n|z_2| \geq |z_1|$. We denote by $\Gamma$ the set $\mathfrak{C}^\times/\sim$, by $\infty$ the equivalence class of $0$, and by $\val:\mathfrak{C} \twoheadrightarrow \Gamma \sqcup \{\infty\}$ the quotient map. Endowing $\Gamma \sqcup \{\infty\}$ with the ordering induced by the reverse ordering on the absolute values of elements of $\mathfrak{C}$, and with the sum induced by the product on $\mathfrak{C}$, makes $\Gamma$ into an ordered abelian group and $\val$ into a valuation. Slightly confusingly, even though this is called the Archimedean valuation, the resulting group $\Gamma$ is not Archimedean: the first-order conditions which describe the growth of the real exponential function compared to polynomials force us to have that if $x \in \mathfrak{R}$ has valuation $\gamma < 0$, then $e^x$ has valuation smaller than $n\gamma$ for every $n \in \mathbb{N}$.
    
    The field $\mathfrak{C}$ is then an algebraically closed non-trivially valued field. By \autocite[Lemma 2.1.15]{MS} the valuation has a splitting, which as in the previous section we denote by $s$ (here also, in the reference $\Gamma \subseteq \mathbb{R}$ is assumed but only to deduce that $\Gamma$ is torsion-free). By composing with the absolute value, we may and will assume that $s(\Gamma) \subseteq \mathfrak{R}_{>0}$.
    
    Since $\mathfrak{R}$ is an elementary extension of $\mathbb{R}$ in the language with $\exp$ and $\sin$, it has an exponential function $\mathfrak{R} \rightarrow \mathfrak{R}_{>0}$, which may be extended to an exponential $\mathfrak{C} \rightarrow \mathfrak{C}^\times$ by the usual formula $\exp(x+iy)=\exp(x)(\cos(y)+i\sin(y))$.  For every $r \in \mathfrak{R}$ we have a well-defined power function on the positive elements defined by $x \mapsto x^r\coloneqq \exp(r\log(x))$, where $\log$ here denotes the (well-defined) inverse of the function $\exp: \mathfrak{R} \rightarrow \mathfrak{R}_{>0}$. If $v = (v_1, \hdots, v_n) \in (\mathfrak{R}_{>0})^n$ and $A = (a_{i,j}) \in \mathrm{Mat}_{d \times n}(\mathfrak{R})$, then we set $v^A=(\Pi_{i=1}^nv_i^{a_{1,i}}, \hdots, \Pi_{i=1}^nv_i^{a_{d,i}})$.

      The valuation ring $\mathcal{O} \subseteq \mathfrak{C}$ consists of those elements whose absolute value lies in the convex hull of $\mathbb{R}$ in $\mathfrak{R}$. The maximal ideal $\mathfrak{m}$ consists of the elements whose absolute value is smaller than any positive real. Any element $z \in \mathcal{O}$ can be written uniquely as $w+\epsilon$, where $w \in \mathbb{C}$ is the complex number closest to $z$ and $\epsilon \in \mathfrak{m}$. The map $z \mapsto w$ is a homomorphism $\mathcal{O} \to \mathbb{C}$ that induces an isomorphism $\mathcal{O}/\mathfrak{m} \simeq \mathbb{C}$. We use this to identify the residue field $\mathcal{O}/\mathfrak{m}$ with $\mathbb{C}$ and write $\res(z)$ for $w$.
      
      It is easy to see that for $z \in \mathcal{O}$ we have $$\exp(z)=\exp(\res(z)+\epsilon)=\exp(\res(z))\exp(\epsilon) \in \exp(\res(z)) \cdot (1+\mathfrak{m}) \subseteq \mathcal{O}^\times,$$ from which it follows that $\exp$ induces a function  $\mathcal{O} \rightarrow \mathcal{O}^\times$ which satisfies $\res(\exp(z))=\exp(\res(z))$.

      Again, we abuse notation and denote also the Cartesian powers of $\mathrm{Re}$, $\val$, $s$, $\exp$, $\log$, and $\res$ by the same symbol respectively.

    \begin{lem}\label{lem:reduction-linear-subspace}
    Let $L \leq \mathbb{G}_{a,\mathfrak{C}}^n$ be a linear subspace and let
    \[ I = \left\{(a_1,\hdots,a_n) \in \mathcal{O}^n \bigg| \sum_{i=1}^{n}{a_ix_i} = 0 \quad \forall (x_1,\hdots,x_n) \in L(\mathfrak{C}) \right\}.\]
    Set $L(\mathcal{O}) = L(\mathfrak{C}) \cap \mathcal{O}^n$. Then
    \[ \res(L(\mathcal{O})) = \left\{(v_1,\hdots,v_n) \in \mathbb{C}^n \bigg| \sum_{i=1}^{n}{b_iv_i} = 0 \quad \forall b = (b_1,\hdots,b_n) \in \res(I)\right\}.\]
    Furthermore, $\dim \res(L(\mathcal{O})) = \dim L$ and $I$ is a free $\mathcal{O}$-module of rank $n- \dim L$.
    \end{lem}

    \begin{proof}
    We induct on $\dim L$, the case $L = \{0\}$ being trivial.

    If $L \neq \{0\}$, then $L(\mathfrak{C})$ contains some non-zero vector $v = (v_1,\hdots,v_n)$. After rescaling and permuting the coordinates, we can assume without loss of generality that $v \in \{1\} \times \mathcal{O}^{n-1}$. There exists an invertible matrix $A \in \gl_n(\mathcal{O})$ such that $Av = (1,0,\hdots,0)^t$.

    Set $L' = A \cdot L$, then $L' = \mathbb{G}_{a,\mathfrak{C}} \times L''$ for some linear subspace $L''$ of $\mathbb{G}_{a,\mathfrak{C}}^{n-1}$ and $L'(\mathcal{O}) = A \cdot L(\mathcal{O})$. Furthermore, if
    \[ I' = \left\{(a_1,\hdots,a_n) \in \mathcal{O}^n \bigg| \sum_{i=1}^{n}{a_ix_i} = 0 \quad \forall (x_1,\hdots,x_n) \in L'(\mathfrak{C}) \right\},\]
    then
    \[ I' = \{a \cdot A^{-1} \mid a \in I\}.\]
    Moreover, $I' \subseteq \{0\} \times \mathcal{O}^{n-1}$ and in fact $I' = \{0\} \times I''$, where
    \[ I'' = \left\{(a_1,\hdots,a_{n-1}) \in \mathcal{O}^{n-1} \bigg| \sum_{i=1}^{n-1}{a_ix_i} = 0 \quad \forall (x_1,\hdots,x_{n-1}) \in L''(\mathfrak{C}) \right\}.\]

    By induction,
    \[ \res(L(\mathcal{O})) = \res(A)^{-1} \cdot \res(L'(\mathcal{O})) = \res(A)^{-1} \cdot (\mathbb{C} \times \res(L''(\mathcal{O})))\]
    is the linear subspace defined by the linear equations with coefficient vectors in the set
    \[ \{(0,b) \cdot \res(A) \mid b \in \res(I'')\} = \{\res(a \cdot A) \mid a \in I'\} = \res(I).\]
     Furthermore, $\dim \res(L(\mathcal{O})) = 1 + \dim \res(L''(\mathcal{O})) = 1 + \dim L'' = \dim L' = \dim L$ and $I$ is isomorphic to $I''$ as an $\mathcal{O}$-module, so free of rank $(n-1) - \dim L'' = n - \dim L$.
    \end{proof}

     If $L \leq \mathbb{G}_{a,\mathfrak{C}}^n$ is a linear subspace, we can therefore denote by $\res(L) \leq \mathbb{G}_{a}^n$ the linear subspace of dimension $\dim L$ that is defined by the residues of the equations defining $L$ whose coefficients all lie in $\mathcal{O}$. We remark that Lemma \ref{lem:reduction-linear-subspace} stays true, by the same proof, if we replace $\mathfrak{C},\mathbb{C}$ and $\mathcal{O}$ by $\mathfrak{R},\mathbb{R}$ and $\mathcal{O} \cap \mathfrak{R}$ respectively.

    We note that while all the results from tropical geometry described in the previous subsection are proved under the hypothesis that $\Gamma$ is Archimedean, we will only apply them to algebraic subvarieties $W \subseteq \mathbb{G}_m^n$, so Remark \ref{rem:constant-coefficients} applies with $K = \mathfrak{C}_0$ for some algebraically closed non-trivially valued subfield $\mathfrak{C}_0$ of $\mathfrak{C}$ with Archimedean value group $\Gamma_0$. Whenever we take the tropicalization of such a $W$ in the following, it will be chosen as in Remark \ref{rem:constant-coefficients}. Because of Lemma \ref{lem:preserve-gr-basis}, any such tropicalization of $W$ is also a tropicalization of $W$ seen as a subvariety defined over the algebraic closure $\overline{\mathbb{C}((\Gamma_0))}$ of the Hahn field, and vice versa. Furthermore, if $\tau$ is a polyhedron in such a tropicalization, then, again because of Lemma \ref{lem:preserve-gr-basis} and Remark \ref{rem:constant-coefficients}, the initial form $W_\tau$, defined via some element of $\relint(\tau)$, does not depend on the ambient field.
    
    Implicitly, when we apply the results from the previous subsection to such varieties $W$, we are considering them as subvarieties of $\mathbb{G}_{m,\mathfrak{C}_0}^n$ and ``lifting'' the relevant results to $\mathfrak{C}$ using model completeness of the theory of algebraically closed non-trivially valued fields of residue characteristic $0$, in the language with sorts for the valued field, residue field, and value group and function symbols for the valuation and a binary variant of the residue map. Model completeness and completeness of this theory are proved in \cite[Theorems 3.4.21 and 4.3.19]{Robinson56}, in the language with a sort for the valued field and a sort for the value group. Since the residue field is the quotient of a definable set by a definable equivalence relation, it is easy to see that adding to the language a sort for the residue field does not affect completeness. Completeness and \cite[Theorem 2.1.1]{HHM} imply that this theory eliminates quantifiers in the three-sorted language, from which model completeness follows. This ``lifting" is sufficiently straightforward that we will never say it explicitly and it will concern only Theorem \ref{thm:structure-theorem}, the fact that $\val(W(\mathfrak{C}))=\supp(\trop(W))$ whenever $W$ is defined over $\mathbb{C}$, and the strengthened form of equality \eqref{eq:ini-var-equality} in Lemma \ref{lem:equal-residue} which we discussed in Remark \ref{rem:equal-residue}.
    
\section{Rotundity and geometrical non-degeneracy}\label{sec:rot-and-geo}

In this section we study the notions of rotundity and geometrical non-degeneracy. In Proposition \ref{prop:characterize-rotund} we give several characterizations of rotundity for algebraic subvarieties of $\mathbb{G}_a^n \times \mathbb{G}_m^n$ of the form $L \times W$, where $L \leq \mathbb{G}_a^n$ is a linear subspace defined over $\mathbb{R}$ and $W \subseteq \mathbb{G}_m^n$ is an algebraic subvariety whose irreducible components all have the same dimension $\codim L$. The main result of this section is Theorem \ref{thm:main-tropical}, in which we show that if we further assume that $W$ is geometrically non-degenerate, then one of the equivalent conditions from Proposition \ref{prop:characterize-rotund} can be strengthened to a version that is uniform in $L$.

Recall from Section \ref{sec:preli} that we are working in an algebraically closed valued field $\mathfrak{C}$, which is obtained as the algebraic closure of a proper elementary extension $\mathfrak{R}$ of $\mathbb{R}_{\exp, \sin}$ and endowed with the Archimedean valuation $\val$. We have also fixed a splitting $s$ of $\val$, $\res$ denotes the residue map, and all function symbols are used also for their Cartesian powers when necessary. 

In this section, the balls we consider as well as their closures are to be understood with respect to the Euclidean norm on the sets of $\mathbb{C}$-points and $\mathfrak{C}$-points of the algebraic groups we work with, so $$|(z_1,\dots,z_n)|=\sqrt{|z_1|^2+\cdots+|z_n|^2}.$$ To avoid ambiguities, we will always use notation of the form $B(z,r)(\mathfrak{C})$ when we consider $\mathfrak{C}$-points. We will sometimes also use the maximum norm $|\cdot|_\infty$ defined by
\[ |(z_1,\dots,z_n)|_\infty = \max_{i = 1, \hdots, n}|z_i|,\]
however, balls are always to be understood with respect to the Euclidean norm. We denote by $\mathbb{S}_1(\mathbb{C})$ the unit circle in $\mathbb{C}$ and by $\mathbb{S}_1$ the corresponding definable set in $\mathbb{R}_{\exp,\sin}$.

As in \cite{Chirka_1989}, a complex analytic set in a complex manifold $M$ is a set $S \subseteq M$ such that for each $x \in S$, there is a neighborhood $U$ of $x$ in $M$ such that $U \cap S$ is the zero set of finitely many holomorphic functions on $U$. So for example an open subset of $\mathbb{C}^n$ is a complex analytic set in $\mathbb{C}^n$.

    \begin{nota}
        For a field $K$, we denote by $G_K(d,n)$ the Grassmannian of linear subspaces $L \leq \mathbb{G}_{a,K}^n$ of dimension $d$. 
    \end{nota}

    \begin{rem}
        We will see the Grassmannian $G_{\mathbb{R}}(d,n)$ as a quotient of the set of $(n-d) \times n$ matrices of rank $n-d$ by the definable equivalence relation of having the same kernel. Then two matrices are equivalent if and only if they define the same subspace. 

        Using definable choice functions for this equivalence relation, which exist by \cite[Proposition 1.2 in Chapter 6]{Dri98} in the semialgebraic structure on $\mathbb{R}$ and so \textit{a fortiori} in $\mathbb{R}_{\exp,\sin}$, we can then see this as a definable set in $\mathbb{R}_{\exp,\sin}$ in which each subspace is represented by a unique matrix. We denote this definable set by $G(d,n)$. Then, $G_{\mathfrak{R}}(d,n) = G(d,n)(\mathfrak{R})$.
        
        A definable open neighborhood of an element $L \in G_{\mathbb{R}}(d,n)$ is the image in $G(d,n)$ of a definable open subset of the set of $(n-d) \times n$ matrices of rank $n-d$ that contains some matrix whose kernel is $L$.
    \end{rem}

    We now prove a version of the Open Mapping Theorem from complex analysis which holds for certain functions over our field $\mathfrak{C}$. This will be used in the proofs of Proposition \ref{prop:characterize-rotund} and Theorem \ref{thm:main-tropical}.

       \begin{lem}\label{lem:open-mapping-in-frakC} 
        Let $L \leq \mathbb{G}_{a, \mathfrak{C}}^n$ be a linear subspace defined over $\mathfrak{R}$, and let $W \subseteq \mathbb{G}_{m,\mathfrak{C}}^n$ be an algebraic subvariety such that every irreducible component of $W$ has codimension $d:=\dim L$. Let $\delta:(L \times W)(\mathfrak{C}) \rightarrow (\mathfrak{C}^\times)^n$ be the map defined by $(\ell,w) \mapsto \frac{w}{\exp(\ell)}$ and fix $\epsilon_0, \epsilon_1 \in \mathbb{R}_{>0}$.

        Suppose that $\ell_0 \in L(\mathfrak{C})$, $w_0 \in W(\mathfrak{C})$ satisfy $|w_0^{-1}|_\infty < \epsilon_1^{-1}$ and
        \[\delta^{-1}(B(z,\epsilon_0)(\mathfrak{C})) \cap  B((\ell_0,w_0),\epsilon_1)(\mathfrak{C}) = \delta^{-1}(B(z,\epsilon_0)(\mathfrak{C})) \cap \overline{B((\ell_0,w_0),\epsilon_1)}(\mathfrak{C}),\]
        where $z = \delta(\ell_0,w_0)$.
        Then
        \[ B(z,\epsilon_0)(\mathfrak{C}) \subseteq \delta((L \times W)(\mathfrak{C})).\]
    \end{lem}

    \begin{proof}
    Let $\mathcal{W}$ be a parametric family, definable over $\mathbb{C}$, of equidimensional algebraic subvarieties of codimension $d$ such that the given $W$ belongs to $\mathcal{W}$. Since $\mathfrak{R}$ is an elementary extension of $\mathbb{R}$ and the statement is first-order expressible, it suffices to prove the lemma for every subvariety in the family over a $\mathbb{C}$-point of the base and every linear subspace $L \in G_{\mathbb{R}}(d,n)$. From now on, we will therefore denote by $W$ an equidimensional subvariety of $\mathbb{G}^n_m$ of codimension $d$, by $L$ a $d$-dimensional linear subspace of $\mathbb{G}^n_a$, defined over $\mathbb{R}$, and by $\delta$ the map $(L \times W)(\mathbb{C}) \to (\ctimes)^n$ defined by $(\ell,w) \mapsto \frac{w}{\exp(\ell)}$.

        Suppose that there exist $\ell_0 \in L(\mathbb{C})$, $w_0 \in W(\mathbb{C})$ satisfying the hypotheses of the lemma and set $z = \delta(\ell_0,w_0)$. We have to show that
         \[ B(z,\epsilon_0) \subseteq \delta((L \times W)(\mathbb{C})).\]
       
         Set $U = \delta^{-1}(B(z,\epsilon_0)) \cap B((\ell_0,w_0),\epsilon_1)$, then $U$ is a non-empty open subset of $(L \times W)(\mathbb{C})$. We want to show that
         \[ \delta(U) = B(z,\epsilon_0).\]
         Since $B(z,\epsilon_0)$ is connected, it is enough to show that $\delta(U) \subseteq B(z,\epsilon_0)$
         is open as well as relatively closed.

         We have by hypothesis that
         \[ \delta(U) = \delta(\overline{B((\ell_0,w_0),\epsilon_1)} \cap (L \times W)(\mathbb{C})) \cap B(z,\epsilon_0).\]
         Now $\overline{B((\ell_0,w_0),\epsilon_1)} \cap (L \times W)(\mathbb{C})$ is compact since $|w_0^{-1}|_\infty < \epsilon_1^{-1}$, so also $\delta(\overline{B((\ell_0,w_0),\epsilon_1)} \cap (L \times W)(\mathbb{C}))$ is compact and hence $\delta(U)$ is relatively closed in $B(z,\epsilon_0)$.

         We claim that $\delta|_U: U \to \mathbb{G}^n_m(\mathbb{C})$ is an open map. Clearly, this will establish that $\delta(U)$ is open. By the Open Mapping Theorem \autocite[p. 107]{grauert2012coherent}, it is enough to show that every fiber of $\delta|_U$ is discrete. Let therefore $A$ be a complex analytic irreducible component of a fiber of $\delta$ over a point in $\delta(U)$ such that $A \cap U \neq \emptyset$. We know that
         \[ A \cap B((\ell_0,w_0),\epsilon_1) = A \cap \overline{B((\ell_0,w_0),\epsilon_1)},\]
         which implies that $A \subseteq B((\ell_0,w_0),\epsilon_1)$ since $A$ is connected. In particular, $A = A \cap \overline{B((\ell_0,w_0),\epsilon_1)}$ is compact. By \cite[Proposition 1 in Section 3.3 of Chapter 1]{Chirka_1989}, this implies that $A$ is finite, so a singleton and we are done.
    \end{proof}

\subsection{Rotundity}\label{subsec:rotund}

    Let $J$ be an algebraic subgroup of $\mathbb{G}^n_m$. We set $TJ = LJ \times J$, where $LJ$ denotes the Lie algebra of $J$. We identify $TJ$ with an algebraic subgroup of $\mathbb{G}_a^n \times \mathbb{G}_m^n$. We denote by $\pi_{J}:\mathbb{G}_m^n \twoheadrightarrow  \mathbb{G}_m^n/J \cong \mathbb{G}_m^{n-\dim J}$ the quotient mapping, and by $\pi_{LJ}$ and $\pi_{TJ}$ the maps defined analogously on $\mathbb{G}_a^n$ and $\mathbb{G}_a^n \times \mathbb{G}^n_m$ respectively.

    \begin{defn}
        Let $V \subseteq \mathbb{G}_a^n \times \mathbb{G}_m^n$ be an irreducible algebraic subvariety. Then $V$ is \textit{rotund} if for every algebraic subgroup $J \leq \mathbb{G}_m^n$, $$\dim \pi_{TJ}(V) \geq n-\dim J.$$
        An arbitrary algebraic subvariety of $\mathbb{G}_a^n \times \mathbb{G}_m^n$ is \textit{rotund} if one of its irreducible components is rotund.
    \end{defn}

    Note that, if $V \subseteq \mathbb{G}_a^n \times \mathbb{G}_m^n$ is rotund and $J \leq \mathbb{G}_m^n$ is an algebraic subgroup, then the Zariski closure of $\pi_{TJ}(V)$ is again rotund.

    The following proposition gives some characterizations of rotundity for  subvarieties of $\mathbb{G}_a^n \times \mathbb{G}_m^n$ which split as a product $L \times W$ for some linear subspace $L \leq \mathbb{G}_a^n$, defined over $\mathbb{R}$, and some algebraic subvariety $W \subseteq \mathbb{G}_m^n$ that is equidimensional of dimension $\codim L$.

    \begin{prop}\label{prop:characterize-rotund}
        Let $L \leq \mathbb{G}^n_{a}$ be a linear subspace defined over $\mathbb{R}$, and let $W \subseteq \mathbb{G}_{m}^n$ be an  algebraic subvariety such that every irreducible component of $W$ has dimension $n-\dim L$.

        The following are equivalent:
        \begin{enumerate}
            \item The subvariety $L \times W$ is rotund. \label{one}
            \item There is a non-empty Zariski open subset $W^\circ$ of $W$ such that for all $w \in \mathbb{G}_m^n(\mathbb{C})$, the set $\Gamma_\exp \cap (L(\mathbb{C}) \times (w^{-1} \cdot W^{\circ}(\mathbb{C})))$ is (empty or) discrete, where $\Gamma_\exp$ denotes the graph of the complex exponential function. \label{two}
            \item There is a point $w_0 \in W(\mathbb{C})$ such that the map $\delta:(L \times W) (\mathbb{C}) \rightarrow (\ctimes)^n$ defined by $(\ell,w) \mapsto \frac{w}{\exp(\ell)}$ is open in a (Euclidean) neighborhood of $(0,w_0)$. \label{three}
            \item There is a non-empty Zariski open subset $W^\circ$ of $W$ such that the map $\delta: (L \times W^\circ)(\mathbb{C}) \rightarrow (\ctimes)^n$ defined by $(\ell,w) \mapsto \frac{w}{\exp(\ell)}$ is open.  \label{four}
            \item $\am + L(\mathbb{R})=\mathbb{R}^n$. \label{five} 
            \item $\bigcup_{\tau \in \trop(W)} \tau(\mathbb{R})+L(\mathbb{R})=\mathbb{R}^n$. \label{six}
            \item There is $\tau \in \trop(W)$ such that $\dim (\tau(\mathbb{R}) + L(\mathbb{R}))=n$. \label{seven}
            \item For all $z=(z_1,\dots,z_n) \in \mathbb{G}_m^n(\mathbb{C})$, there exist $\epsilon>0$ and $z'=(z_1',\dots,z_n') \in \mathbb{G}_m^n(\mathbb{C})$ such that $|z_i'| = |z_i|$ for all $i$ and the open ball $B(z', \epsilon)$ with respect to the Euclidean metric on $\mathbb{G}^n_m(\mathbb{C})$ is contained in 
            $\frac{W(\mathbb{C})}{\exp(L(\mathbb{C}))}$. \label{eight}
        \end{enumerate}
    \end{prop}

    \begin{proof}
        For $(\ref{one} \Rightarrow \ref{two})$, we use \autocite[proof of Proposition 6.2 and Remark 6.3]{K19} and \autocite[proof of Proposition 3.7]{Gal23}, applying these to some irreducible component $W_{\textnormal{Irr}}$ of $W$ for which $L \times W_{\textnormal{Irr}}$ is rotund. Strictly speaking, the latter reference is concerned with openness of the $\delta$-map, but the same argument yields discreteness of its fibers. For $(\ref{two} \Rightarrow \ref{four})$, assuming $W^\circ$ as in $(\ref{two})$ given, the same $W^\circ$ satisfies $(\ref{four})$ by the Open Mapping Theorem \autocite[p. 107]{grauert2012coherent}. $(\ref{four} \Rightarrow \ref{three})$ is obvious. $(\ref{three} \Rightarrow \ref{one})$ is \autocite[Proposition 3.8]{Gal23} (note that any Euclidean neighborhood of $(0,w_0)$ will contain a Euclidean neighborhood of some $(0,w_1)$, where $w_1$ does not lie in the intersection of two distinct irreducible components of $W$).
        
        We show $(\ref{one} \Rightarrow \ref{five}, \ref{eight})$.

        Assume then that $L \times W$ is rotund, and let $W^\circ$ be the Zariski open subset given by $(\ref{four})$. Let $W':=W \setminus W^\circ$. Let $F$ be some non-zero algebraic function on $W$ which vanishes on all points of $W'$.

        Consider the algebraic subvariety $W_1$ of $\mathbb{G}_m^{n+1}$ whose set of $\mathbb{C}$-points is $$\{(w_1,\dots,w_{n+1}) \in (\ctimes)^{n+1} \mid (w_1,\dots,w_n) \in W(\mathbb{C}) \wedge w_{n+1}=F(w_1,\dots,w_n) \}.$$ Then $(L \times \mathbb{G}_a) \times W_1$ satisfies Condition (\ref{four}) with $W_1^\circ=W_1$. Hence we may apply \cite[Lemma 6.15]{Gal23} to all translates $z \cdot W_1$ for $z \in (\ctimes)^{n+1}$, which gives that $\mathcal{A}_{W_1}+(L(\mathbb{R}) \times \mathbb{R})=\mathbb{R}^{n+1}$, from which it follows that $L \times W$ satisfies Condition (\ref{five}). It also follows that $(L \times \mathbb{G}_a) \times W_1$ satisfies (\ref{eight}). But then $L \times W$ also does.
        
        For $(\ref{eight} \Rightarrow \ref{three})$, note that, because of \cite[Proposition on p. 41 in Section 3.8 of Chapter 1]{Chirka_1989}, $(\ref{eight})$ implies that the map $\delta: (L \times W)(\mathbb{C}) \to (\mathbb{C}^\times)^n$ defined by $(\ell,w) \mapsto \frac{w}{\exp(\ell)}$ has a fiber with an isolated point and this in turn implies that $\delta$ is open in a neighborhood of some $(\ell_0,w_0) \in (L \times W)(\mathbb{C})$ by the Open Mapping Theorem and the upper semicontinuity of the fiber dimension of a holomorphic map \cite[Proposition 49.14]{Kaup_Kaup_1983}. For any other $\ell_1 \in L(\mathbb{C})$, there is a neighborhood $U$ of $(\ell_1,w_0)$ such that $\delta^{\ast}: (L \times W)(\mathbb{C}) \rightarrow (\ctimes)^n$ defined by $\delta^{\ast}(\ell,w)=\delta(\ell+\ell_0-\ell_1,w) = \frac{w}{\exp(\ell+\ell_0-\ell_1)}=\exp(\ell_1-\ell_0) \delta(\ell,w)$ is open when restricted to $U$. Hence $(\ref{three})$ holds. 
        
        $(\ref{six} \Rightarrow \ref{seven})$ is obvious. We now show $(\ref{five} \Rightarrow \ref{six})$. 

        Let $x \in \mathbb{R}^n$. We have to show that $x \in \supp(\trop(W))(\mathbb{R})+L(\mathbb{R})$. Since $0 \in \supp(\trop(W))$, if $x \in L(\mathbb{R})$ we are done, so we assume $x \notin L(\mathbb{R})$. Consider the notion of Kuratowski convergence induced on the set of subsets of $\mathbb{R}^n$ by the Euclidean metric on $\mathbb{R}^n$; see for example \cite[Section 29]{Kur}. With respect to this notion, for every infinite subset $\mathcal{N} \subseteq \mathbb{N}$, the sequence of subspaces $\{jx+L(\mathbb{R})\}_{j \in \mathcal{N}}$ of $\mathbb{R}^n$ has empty limit. By assumption, for every $j \in \mathbb{N}$ there is $a_j \in \am$ such that $jx+ L(\mathbb{R})=a_j + L(\mathbb{R})$. If the sequence $\{a_j\}_{j \in \mathbb{N}}$ has an accumulation point $a \in \mathbb{R}^n$, then $\lim_{j \in \mathcal{N}} jx+L(\mathbb{R})=a+L(\mathbb{R})$ for some infinite subset $\mathcal{N} \subseteq \mathbb{N}$, which yields a contradiction with the observation above. Therefore, we have that $\lim_{j \to \infty} |a_j| = \infty$.

      Denote by $S_{n-1}$ the unit ($n-1$)-sphere in $\mathbb{R}^n$. There exists an infinite subset $\mathcal{N} \subseteq \mathbb{N}$ such that $ja_j \neq 0$ for all $j \in \mathcal{N}$ and $ \left\{\frac{a_j}{|a_j|}\right\}_{j \in \mathcal{N}}$ has a limit $b \in S_{n-1}$. Then $\mathbb{R}_{\geq 0} \cdot b= \mathbb{R}_{\geq 0} \cdot \lim_{j \in \mathcal{N}} \frac{a_j}{|a_j|} $, so by \cite[Theorem A]{Jon}, the fact that $\lim_{j \to \infty}{|a_j|} = \infty$, and the fact that $\trop(W)$ is a fan, we must have $\mathbb{R}_{\geq 0} \cdot b \subseteq \supp(\trop(W))(\mathbb{R})$. Writing $x=\frac{1}{j} jx$, for every $j \in \mathcal{N}$, we see
        \[    x \in \frac{1}{j} \left(|a_j| \frac{a_j}{|a_j|} +L(\mathbb{R}) \right)=\frac{|a_j|}{j} \frac{a_j}{|a_j|}+L(\mathbb{R})\] and thus
        \[ x \in \mathbb{R}_{\geq 0} \cdot \frac{a_j}{|a_j|} + L(\mathbb{R}).\]

        Since the sequence $\left\{ \mathbb{R}_{\geq 0} \cdot \frac{a_j}{|a_j|} \right\}_{j \in \mathcal{N}}$ converges to the half-line $\mathbb{R}_{\geq 0} \cdot b$, we have that 
        \[ \lim_{j \in \mathcal{N}} \left(  \mathbb{R}_{\geq 0} \cdot \frac{a_j}{|a_j|} + L(\mathbb{R}) \right)= \mathbb{R}_{\geq 0} \cdot b + L(\mathbb{R}). \]

        Hence $x \in \mathbb{R}_{\geq 0} \cdot b + L(\mathbb{R}) \subseteq \supp(\trop(W))(\mathbb{R})+L(\mathbb{R})$.
       
        Finally, let us prove that $(\ref{seven} \Rightarrow \ref{one})$ (more precisely, we prove the contrapositive $(\neg \ref{one} \Rightarrow \neg \ref{seven})$).
        
        \textbf{Claim}: If $W$ has an initial variety $W_\tau$ such that $L \times W_\tau$ is rotund, then $L \times W$ is rotund. 

        \textbf{Proof of Claim}: Assume that $L \times W_\tau$ is rotund for some $\tau \in \trop(W)$, and let $\gamma \in \relint(\tau)$ and $a \coloneqq s(\gamma)$. By Lemma \ref{lem:componentsofinitialform}, every irreducible component of $W_\tau$ has dimension $\dim W = n - \dim L$. Let $W_\tau^\circ$ be the non-empty Zariski open subset of $W_\tau$ obtained by applying (\ref{two}) to $L \times W_\tau$. By Lemma \ref{lem:equal-residue} and Remark \ref{rem:equal-residue}, $W_\tau(\mathbb{C})=\res((a^{-1} \cdot W(\mathfrak{C})) \cap (\mathcal{O}^\times)^n)$.

        We now choose $w_0 \in W_\tau^\circ(\mathbb{C})$ and $w_1 \in (a^{-1} \cdot W(\mathfrak{C})) \cap (\mathcal{O}^\times)^n$ such that $\res(w_1) = w_0$.

         By Condition (\ref{two}), there exists $\epsilon_1 \in \mathbb{R}_{>0}$ such that $|w_0^{-1}|_\infty < \epsilon_1^{-1}$ and, if $(\ell,w) \in (L \times W_\tau)(\mathbb{C}) \cap \overline{B((0,w_0),\epsilon_1)}$ with $w/\exp(\ell) = w_0$, then $(\ell,w) = (0,w_0)$. Since $\res(|w_1^{-1}|_\infty) = |w_0^{-1}|_\infty < \epsilon_1^{-1}$, we have that 
        \[(L \times a^{-1} \cdot W)(\mathfrak{C}) \cap \overline{B((0,w_1),\epsilon_1)}(\mathfrak{C}) \subseteq \mathcal{O}^{n} \times (\mathcal{O}^\times)^{n}.\]
         
         There then exists $\epsilon_0 \in \mathbb{R}_{>0}$ such that
        \[
        \min_{(\ell,w) \in (L \times W_\tau)(\mathbb{C}) \cap \partial B((0,w_0),\epsilon_1)}{|w/\exp(\ell)-w_0|} > \epsilon_0.
        \]
        This implies that
        \[
        \min_{(\ell,w) \in (L \times a^{-1} \cdot W)(\mathfrak{C}) \cap \partial B((0,w_1),\epsilon_1)(\mathfrak{C})}{|w/\exp(\ell)-w_1|} > \epsilon_0,
        \] where the latter minimum exists as it is the minimum of a function which belongs to a family, definable over $\mathbb{C}$, of definable and continuous functions with closed and bounded domains. 
        
        Thus the map $\delta: (L \times a^{-1} \cdot W)(\mathfrak{C}) \rightarrow (\mathfrak{C}^\times)^n$ defined by $(\ell,w) \mapsto w/\exp(\ell)$ satisfies
        \[ \delta^{-1}(B(w_1,\epsilon_0)(\mathfrak{C})) \cap B((0,w_1),\epsilon_1)(\mathfrak{C}) = \delta^{-1}(B(w_1,\epsilon_0)(\mathfrak{C})) \cap \overline{B((0,w_1),\epsilon_1)}(\mathfrak{C}).\]
        Recalling that $\res(|w_1^{-1}|_{\infty}) < \epsilon_1^{-1}$, we conclude that the assumptions of Lemma \ref{lem:open-mapping-in-frakC} are satisfied for the point $(0,w_1)$ of $L \times a^{-1} \cdot W$ and the radii $\epsilon_0, \epsilon_1$.
        
        We then have that $\frac{a^{-1} \cdot W}{\exp(L)}(\mathfrak{C})$ has non-empty interior. Multiplying this by $a$, we get that $\frac{W}{\exp(L)}(\mathfrak{C})$ has non-empty interior; since $\frac{W}{\exp(L)}$ is defined over $\mathbb{C}$, and having non-empty interior is a first-order condition, $\frac{W}{\exp(L)}(\mathbb{C})$ also has non-empty interior. Hence, $L \times W$ is rotund by Condition $(\ref{three})$, again because of \cite[Proposition on p. 41 in Section 3.8 of Chapter 1]{Chirka_1989}, the upper semicontinuity of the fiber dimension of a holomorphic map, and the Open Mapping Theorem. This completes the proof of the above claim.
    
        We now show that, if $L \times W$ is not rotund, then $$\dim (\tau(\mathbb{R})  + L(\mathbb{R})) \neq n$$ for all $\tau \in \trop(W)$. Since $L \times W$ is not rotund and $\dim L + \dim W = n$, we must have $\dim L < n$ and so $\dim W > 0$. Let $\tau \in \trop(W)$. We want to show that $\dim(\tau(\mathbb{R})+L(\mathbb{R})) \neq n$. We may assume that $\dim \tau > 0$. By assumption and the above claim, $L \times W_\tau$ is not rotund. Let $J_\tau$ be as in Definition \ref{def:jtau}. Recall that $LJ_\tau$ denotes the Lie algebra of $J_\tau$. By Proposition \ref{prop:invariant}, $J_\tau$ stabilizes every irreducible component of $W_\tau$.
        
        Let $W_{\tau,\mathrm{Irr}}$ be an irreducible component of $W_\tau$. Since $L \times W_{\tau,\mathrm{Irr}}$ is not rotund, there exists an algebraic subgroup $J \leq \mathbb{G}^n_m$ such that $\dim \pi_{TJ}(L \times W_{\tau,\mathrm{Irr}}) < n - \dim J$. Set $\widetilde{J} = J \cdot J_\tau$. Then $\pi_J(\widetilde{J})$ stabilizes the Zariski closure of $\pi_J(W_{\tau,\mathrm{Irr}})$ and so
        \begin{align*} \dim \pi_{T\widetilde{J}}(L \times W_{\tau,\mathrm{Irr}}) & \leq \dim \pi_{TJ}(L \times W_{\tau,\mathrm{Irr}}) - \dim \pi_J(\widetilde{J}) \\
        &< n - \dim J - \dim \pi_J(\widetilde{J}) \\ & = n - \dim \widetilde{J}.
        \end{align*}
        But $\pi_{T\widetilde{J}}(L \times W_{\tau,\mathrm{Irr}}) = \pi_{T(\widetilde{J}/J_\tau)}(\pi_{TJ_\tau}(L \times W_{\tau,\mathrm{Irr}}))$ and so $\pi_{TJ_\tau}(L \times W_{\tau,\mathrm{Irr}})$ is not rotund. 

        In particular, $\pi_{LJ_\tau}(L) \neq \mathbb{G}_a^{n-\dim J_\tau}$ and so $\dim(\tau(\mathbb{R}) + L(\mathbb{R})) \neq n$. This completes the proof.
    \end{proof}
  
    \subsection{Geometrical non-degeneracy and uniformity}\label{sec:geonondeg}

Recall that an irreducible algebraic subvariety $W \subseteq \mathbb{G}_m^n$ is geometrically non-degenerate if for every algebraic subgroup $J \leq \mathbb{G}_m^n$, $$\dim \pi_J(W) = \min \{\dim W, n- \dim J\}.$$  It is straightforward to verify that if $W \subseteq \mathbb{G}_m^n$ is geometrically non-degenerate and of codimension $d$, then $L \times W$ is rotund for every $d$-dimensional subspace $L$ of $\mathbb{G}_a^n$ defined over $\mathbb{R}$.  

Our goal in this section is to show that geometrical non-degeneracy implies a uniform version of Condition (\ref{eight}) in Proposition \ref{prop:characterize-rotund}: there is an $\epsilon >0$ such that for all $z \in (\ctimes)^n$ and every subspace $L$ of $\mathbb{G}_a^n$ defined over $\mathbb{R}$ of dimension equal to $\codim W$, there is $s_1 \in \mathbb{S}_1^n(\mathbb{C})$ such that $$B(s_1, \epsilon) \subseteq \frac{z^{-1} \cdot W(\mathbb{C})}{\exp(L(\mathbb{C}))}.$$

        \begin{lem}\label{lem:rotundity-face}
    Let $L \leq \mathbb{G}_{a}^n$ be a linear subspace of dimension $d$ defined over $\mathbb{R}$ and let $W \subseteq \mathbb{G}_m^n$ be an algebraic subvariety whose irreducible components all have codimension $d$. Let $\tau \in \trop(W)$. Then, $L \times W_\tau$ is rotund if and only if there exists some $\sigma \in \trop(W)$ such that $\tau \subseteq \sigma$ and $\dim(\sigma(\mathbb{R}) + L(\mathbb{R})) = n$. In particular, if $L \times W_\tau$ is rotund, then $L \times W_{\tau'}$ is rotund for every $\tau' \in \trop(W)$ such that $\tau' \subseteq \tau$.  
    \end{lem}

    \begin{proof}
    Note that every irreducible component of $W_\tau$ has codimension $d$ by Lemma \ref{lem:componentsofinitialform}. If $L \times W_\tau$ is rotund, then, by Propositions \ref{prop:star-is-tropical} and  \ref{prop:characterize-rotund}(\ref{seven}), there exists $\widetilde{\sigma} \in \star_{\trop(W)}(\tau)$ such that $\dim(\widetilde{\sigma}(\mathbb{R}) + L(\mathbb{R})) = n$. By the definition of the star, there exists $\sigma \in \trop(W)$ such that $\tau \subseteq \sigma$ and $\lin(\widetilde{\sigma}) = \lin(\sigma)$ and therefore $\dim(\sigma(\mathbb{R}) + L(\mathbb{R})) = n$.

    Conversely, if such a $\sigma$ exists, then the corresponding $\widetilde{\sigma} \in \star_{\trop(W)}(\tau)$ satisfies $\dim(\widetilde{\sigma}(\mathbb{R}) + L(\mathbb{R})) = n$ and then the same Propositions \ref{prop:star-is-tropical} and  \ref{prop:characterize-rotund}(\ref{seven}) imply that  $L \times W_\tau$ is rotund.

    The ``in particular'' statement then follows.
    \end{proof}

    \begin{lem}\label{lem:neighborhood-in-grass}
        Let $L_0 \leq \mathbb{G}_{a}^n$ be a linear subspace of dimension $d$ defined over $\mathbb{R}$ and $W \subseteq \mathbb{G}_{m}^n$ an algebraic subvariety whose irreducible components all have codimension $d$ such that $L_0 \times W$ is rotund.

        Set $S:=\{\tau \in \trop(W) \mid L_0 \times W_\tau \text{ is rotund}\}$ and $T:=\bigcup_{\tau \in S} \tau$. There is a definable open neighborhood $U$ of $L_0$ in the Grassmannian $G(d,n)$ such that for all $L \in U(\mathbb{R})$, $T(\mathbb{R})+L(\mathbb{R})=\mathbb{R}^n$.
    \end{lem}

    \begin{proof}
        By Lemma \ref{lem:rotundity-face}, we have that $\dim(\tau(\mathbb{R}) + L_0(\mathbb{R})) = n$ for all $\tau \in S$ that are maximal with respect to inclusion. We can then find a definable open neighborhood $U \subseteq G(d,n)$ of $L_0$ in the Grassmannian such that for all $L \in U(\mathbb{R})$ and all such $\tau$, we have that $\dim(\tau(\mathbb{R}) + L(\mathbb{R})) = n$. By Lemma \ref{lem:rotundity-face}, $S$ is closed under taking non-empty faces and therefore also a polyhedral complex.

        It follows from Proposition \ref{prop:characterize-rotund}(\ref{six})
         and Lemma \ref{lem:rotundity-face} that $\mathbb{R}^n \backslash (T(\mathbb{R}) + L_0(\mathbb{R}))$ is contained in a finite union of linear subspaces of dimension $< n$. Its complement $T(\mathbb{R})+L_0(\mathbb{R})$ is closed because it is a finite union of closed polyhedra. We deduce that
        \[ T(\mathbb{R}) + L_0(\mathbb{R}) = \mathbb{R}^n.\]

        We now prove the lemma by induction on $\dim W$, the case $\dim W = 0$ being trivial.

        If $\dim W=1$, then $d=n-1$, so $L_0(\mathbb{R})$ is a hyperplane in $\mathbb{R}^n$, say defined by an equation $\sum_{i=1}^n a_ix_i=0$ with $a_1,\dots,a_n \in \mathbb{R}$. It follows from $T(\mathbb{R})+L_0(\mathbb{R}) = \mathbb{R}^n$ that there exist $(y_1,\dots,y_n)$ and $(z_1,\dots,z_n)$ in $T(\mathbb{R})$ such that 
        \[\sum_{i=1}^n a_iy_i < 0 < \sum_{i=1}^n a_i z_i.\] These inequalities stay true if we replace $(a_1,\dots,a_n)$ with a vector $(b_1,\dots,b_n)$ that is sufficiently close to $(a_1,\dots,a_n)$, and since every $\tau \in S$ is a cone, that is enough to deduce that $T(\mathbb{R})+L(\mathbb{R})=\mathbb{R}^n$ for all $L \in U(\mathbb{R})$ after maybe replacing $U$ by some smaller definable open neighborhood of $L_0$.

        Assume now $\dim W > 1$, and fix $L \in U(\mathbb{R})$. Then $d < n-1$, from which it follows that $\mathbb{R}^n \setminus L(\mathbb{R})$ is connected. We have that $(T(\mathbb{R})+L(\mathbb{R}))\setminus L(\mathbb{R})$ is relatively closed in $\mathbb{R}^n \setminus L(\mathbb{R})$ because $T(\mathbb{R})+L(\mathbb{R})$ is a finite union of closed polyhedra. Since $0 \in T(\mathbb{R})$, it remains to show that $(T(\mathbb{R})+L(\mathbb{R}))\setminus L(\mathbb{R})$ is open, as we can then conclude that $T(\mathbb{R}) + L(\mathbb{R}) = \mathbb{R}^n$ by connectedness of $\mathbb{R}^n \setminus L(\mathbb{R})$.

        Let $x \in T(\mathbb{R})+L(\mathbb{R})$, $x \notin L(\mathbb{R})$. We want to show that $x$ lies in the interior of $(T(\mathbb{R})+L(\mathbb{R}))\setminus L(\mathbb{R})$. There is a polyhedron $\tau_0 \in S$ such that $x \in y+L(\mathbb{R})$ for some $y \in \relint(\tau_0)(\mathbb{R})$. Since $\trop(W)$ is a fan and $x \notin L(\mathbb{R})$, the dimension of $\tau_0$ is positive.
        
        By the definition of $S$, $L_0 \times W_{\tau_0}$ is rotund. By Proposition \ref{prop:star-is-tropical}, we may take $\trop(W_{\tau_0}) = \star_{\trop(W)}(\tau_0)$. Set 
        \[S_0:=\{ \tau \in \trop(W_{\tau_0}) \mid L_0 \times (W_{\tau_0})_\tau \textnormal{ is rotund} \} \] and set $T_0:=\bigcup_{\tau \in S_0} \tau$.

        We claim that $S_0 = \star_S(\tau_0)$. This follows from Lemma \ref{lem:rotundity-face} since $\lin(\tau^\ast) = \lin(\tau)$ and $\tau \subseteq \sigma \Leftrightarrow \tau^\ast \subseteq \sigma^\ast$ for all $\tau, \sigma \in \trop(W)$ such that $\tau_0 \subseteq \tau \cap \sigma$, where $\tau^\ast, \sigma^\ast$ denote the polyhedra in $\trop(W_{\tau_0}) = \star_{\trop(W)}(\tau_0)$ that correspond to $\tau, \sigma$ respectively. Then by the definition of the star, for all $y' \in \relint(\tau_0)$ we have that $S_0$ and $S-y'$ locally coincide around $0$. Hence, $S_0 + y'$ and $S$ locally coincide around $y'$, but $S_0 + y' = S_0$ since every element of $S_0$ is invariant under $\lin(\tau_0)$ by Remark \ref{rem:invariant-star} and $\tau_0$ is a cone, so it is contained in $\lin(\tau_0)$. Therefore, if we consider $\relint(\tau_0)$, $T_0$, and $T$ as definable sets in the language of ordered $\mathbb{Q}$-vector spaces (without parameters), we see that the formula
        \[ \forall y' \in \relint(\tau_0)(\exists a > 0(\forall x' \in y' + (-a,a)^n(x' \in T \leftrightarrow x' \in T_0)))\]
        holds in $\Gamma$ and therefore it holds in $\mathbb{R}$ by completeness of the theory of non-trivial ordered $\mathbb{Q}$-vector spaces (see \cite[Remark 7.9 in Chapter 1]{Dri98}). Hence there is a neighborhood of $y$ in $\mathbb{R}^n$ in which the sets $T(\mathbb{R})$ and $T_0(\mathbb{R})$ coincide.

        Set $Q\coloneqq \lin(\tau_0)(\mathbb{C})$ and let $\pi_Q$ and $\pi_{\exp(Q)}$ denote the usual projections associated to $Q$. Observe that, in the notation of Proposition \ref{prop:invariant}, $\exp(Q)=J_{\tau_0}(\mathbb{C})$, and therefore, by that proposition, $\exp(Q)$ stabilizes every irreducible component $W_{\tau_0,\mathrm{Irr}}$ of $W_{\tau_0}$. In particular, each $\pi_{\exp(Q)}(W_{\tau_0,\mathrm{Irr}})$ is Zariski closed and of dimension $\dim W_{\tau_0,\mathrm{Irr}} - \dim Q = \dim W - \dim Q$ by Lemma \ref{lem:componentsofinitialform}. Then $\pi_Q(L_0) \times \pi_{\exp(Q)}(W_{\tau_0})$ is rotund by rotundity of $L_0 \times W_{\tau_0}$, so in particular
        \begin{equation}\label{eq:dim}
            \dim \pi_Q(L_0) + \dim \pi_{\exp(Q)}(W_{\tau_0}) = \dim L_0 + \dim W - \dim Q = n - \dim Q.
        \end{equation}
        Since $y \in \relint(\tau_0)(\mathbb{R})$, we have that $\pi_Q(y)=0$.
        
        We now want to apply the induction hypothesis to $\pi_Q(L_0) \times \pi_{\exp(Q)}(W_{\tau_0})$. Let $S_1$ denote the set
         \[ \{ \tau \in \trop(\pi_{\exp(Q)}(W_{\tau_0})) \mid \pi_Q(L_0) \times (\pi_{\exp(Q)}(W_{\tau_0}))_\tau \text{ is rotund}\}.\]
         It follows from Lemma \ref{lem:rotundity-face} that a point $r \in \mathbb{R}^n$ belongs to $T_0(\mathbb{R})$ if and only if $r \in \sigma(\mathbb{R})$ for some $\sigma \in \trop(W_{\tau_0})$ such that $\dim(\sigma(\mathbb{R})+L_0(\mathbb{R})) = n$. Because of Proposition \ref{prop:star-is-tropical}, Remark \ref{rem:invariant-star}, and \cite[Corollary 3.2.13]{MS}, this is the case if and only if $\pi_Q(r) \in \sigma'(\mathbb{R})$ for some $\sigma' \in \trop(\pi_{\exp(Q)}(W_{\tau_0}))$ such that $\dim(\sigma'(\mathbb{R})+\pi_Q(L_0)(\mathbb{R})) = n - \dim Q$ (the subdivision of $\supp(\trop(\pi_{\exp(Q)}(W_{\tau_0})))(\mathbb{R}) = \pi_Q(\supp(\trop(W_{\tau_0}))(\mathbb{R}))$ into polyhedra does not matter here). Again by Lemma \ref{lem:rotundity-face}, this is equivalent to $\pi_Q(r)$ lying in $\bigcup_{\tau \in S_1}{\tau(\mathbb{R})}$. Using the surjectivity of $\pi_Q|_{\mathbb{R}^n}: \mathbb{R}^n \to \mathbb{R}^n/(Q \cap \mathbb{R}^n)$, we conclude that
         \[ \bigcup_{\tau \in S_1}{\tau(\mathbb{R})} = \pi_Q(T_0(\mathbb{R})).\]
  
        Note that $\dim \pi_Q(L_0) = \dim L_0$ by \eqref{eq:dim} and that the preimage of a definable open neighborhood of $\pi_Q(L_0)$ under the map between Grassmannians induced by $\pi_Q|_{\mathbb{R}^n}$ contains a definable open neighborhood of $L_0$. By the induction hypothesis, after maybe replacing $U$ by some smaller definable open neighborhood of $L_0$ (this can be done independently of $\tau_0$ as $\tau_0$ varies in the finite set $\trop(W)$) and recalling that $L \in U(\mathbb{R})$, we then know that
        \[\bigcup_{\tau \in S_1}{\tau(\mathbb{R})} + \pi_Q(L)(\mathbb{R}) = \mathbb{R}^n/(Q \cap \mathbb{R}^n),\] which is equal to $\pi_Q(T_0(\mathbb{R}))+\pi_Q(L)(\mathbb{R})$ by the above. From this, it follows that
        \[\mathbb{R}^n = (\pi_Q|_{\mathbb{R}^n})^{-1}(\pi_Q(T_0(\mathbb{R}))+\pi_Q(L)(\mathbb{R})) = T_0(\mathbb{R})+L(\mathbb{R}),\]
        where the second equality holds due to Proposition \ref{prop:star-is-tropical} and Remark \ref{rem:invariant-star}, and so $y$ lies in the interior of $T_0(\mathbb{R})+L(\mathbb{R})$ (inside $\mathbb{R}^n$).
        
        For each $\tau \in S_0$ that is maximal with respect to inclusion, it follows from the fact that $S_0 = \star_S(\tau_0)$ and our assumption on $U$ from the beginning of the proof that $\lin(\tau)(\mathbb{R}) + L(\mathbb{R}) = \mathbb{R}^n$. Theorem \ref{thm:structure-theorem} and Lemma \ref{lem:componentsofinitialform} then imply that
        \[\dim (\lin(\tau)(\mathbb{R})) = \dim W_{\tau_0} = \dim W = n-d =\codim(L(\mathbb{R}))\]
       and therefore the addition map $\varphi_\tau: \lin(\tau)(\mathbb{R}) \times L(\mathbb{R}) \to \mathbb{R}^n$ is a homeomorphism. Since $T_0(\mathbb{R})$ and $T(\mathbb{R})$ coincide in a neighborhood of $y$, we deduce that $y$ also lies in the interior of $T(\mathbb{R}) + L(\mathbb{R})$, and therefore so does $x$. Hence $(T(\mathbb{R})+L(\mathbb{R})) \setminus L(\mathbb{R})$ is open as required. 
    \end{proof}

    \begin{prop}\label{prop:non-archimedean-version}
        Let $L \leq \mathbb{G}_{a,\mathfrak{C}}^n$ be a linear subspace of dimension $d$ defined over $\mathfrak{R}$ and $W \subseteq \mathbb{G}_{m}^n$ an algebraic subvariety such that $W$ is geometrically non-degenerate and $\dim L + \dim W = n$.

        Then there is a finite set of polyhedra $S \subseteq \trop(W)$ such that the following hold:
        \begin{itemize}
            \item[$(1)$] For all $z \in (\mathfrak{C}^\times)^n$ there is $z' \in (z \cdot \exp(L(\mathfrak{C})) \cdot (\mathcal{O}^\times)^n) \cap W(\mathfrak{C})$ such that $\val(z') \in \tau$ for some $\tau \in S$.
            \item[$(2)$] $\res(L) \times W_\tau$ is rotund for all $\tau \in S$.
        \end{itemize}
    \end{prop}

    \begin{proof}
        Since $W$ is geometrically non-degenerate, $\res(L) \times W$ is rotund. Set $S=\{\tau \in \trop(W) \mid \res(L) \times W_{\tau} \textnormal{ is rotund}\}$ and $T=\bigcup_{\tau \in S} \tau$. By Lemma \ref{lem:neighborhood-in-grass}, there exists a definable open neighborhood $U \subseteq G(d,n)$ of $\res(L)$ such that for all $L' \in U(\mathbb{R})$, we have that $T(\mathbb{R}) + L'(\mathbb{R}) = \mathbb{R}^n$.

        Therefore the first-order sentence 
        \[\forall L' \in U \left( \forall (x_1,\dots,x_n) (\exists t \in T, \ell \in L' (t+\ell=(x_1,\dots,x_n)))\right)  \] is true in the structure $\mathbb{R}_{\exp,\sin}$; it has parameters in $\mathbb{R}$ only, so it is also true in $\mathfrak{R}$. Since $L \in U(\mathfrak{R})$ by Lemma \ref{lem:reduction-linear-subspace}, we have that $T(\mathfrak{R})+L(\mathfrak{R})=\mathfrak{R}^n$, where $T(\mathfrak{R}) = \bigcup_{\tau \in S}{\tau(\mathfrak{R})}$ and for each $\tau \in S$, $\tau(\mathfrak{R})$ is the set of points in $\mathfrak{R}^n$ which satisfy the semi-linear inequalities defining $\tau$.

         Let now $z \in (\mathfrak{C}^\times)^n$ and set $\gamma = \val(z) \in \Gamma^n$. Let $a \in (\mathfrak{R}_{>0})^n$ with $\val(a) = \gamma$. Then $\log(a) \in \mathfrak{R}^n$ and so there are $\tau \in S$, $t \in \tau(\mathfrak{R})$, and $\ell \in L(\mathfrak{R})$ such that $t-\ell=-\log(a)$. It follows that $\exp(-t)\exp(\ell)=a$, and $-\val(\exp(t))+\val(\exp(\ell))=\gamma$. Since $\exp:(\mathfrak{R},+) \rightarrow (\mathfrak{R}_{>0},\cdot)$ and $-\val:(\mathfrak{R}_{>0},\cdot) \rightarrow (\Gamma,+)$ are both homomorphisms of ordered abelian groups and $\tau$ is a cone defined by homogeneous semi-linear inequalities with integer coefficients, we have that $-\val(\exp(t)) \in \tau$. 
        This means that $z_0 = z \cdot \exp(-\ell) \in z \cdot \exp(L(\mathfrak{C}))$ satisfies
        \[\val(z_0)=\gamma - \val(\exp(\ell)) = -\val(\exp(t)) \in \tau.\]
        We set $\alpha = \val(z_0)$.

        Since $\alpha \in \tau$, by definition of the tropicalization, there exists $z' \in W(\mathfrak{C})$ with $\val(z') = \alpha$. It follows that $\val\left(\frac{z_0}{z'}\right) = \val\left(\frac{z'}{z_0}\right) = 0$ and so $\frac{z'}{z_0} \in (\mathcal{O}^\times)^n$. Therefore
        \[ z' = \frac{z'}{z_0} \cdot z_0 \in z \cdot \exp(L(\mathfrak{C})) \cdot (\mathcal{O}^\times)^n\]
        and we are done.
    \end{proof}

    \begin{thm}\label{thm:main-tropical}
        Let $W \subseteq \mathbb{G}_{m}^n$ be an algebraic subvariety of codimension $d$ that is geometrically non-degenerate. 
        
        There is $\epsilon \in \mathbb{R}_{>0}$ such that for every linear subspace $L \in G_{\mathbb{R}}(d,n)$ and every $z \in (\ctimes)^n$ there is $s_1 \in \mathbb{S}_1^n(\mathbb{C})$ such that $B(s_1,\epsilon) \subseteq \frac{z^{-1} \cdot W(\mathbb{C})}{\exp(L(\mathbb{C}))}$.
    \end{thm}

    \begin{proof}
        Let $L \in G_{\mathfrak{R}}(d,n)$. We will show that for all $z \in (\mathfrak{C}^\times)^n$, there are $s_1 \in \mathbb{S}_1^n(\mathfrak{C})$ and $\epsilon \in \mathbb{R}_{>0}$ such that $B(s_1, \epsilon)(\mathfrak{C}) \subseteq \frac{z^{-1} \cdot W(\mathfrak{C})}{\exp(L(\mathfrak{C}))}$. It is important to notice that $\epsilon$ is in $\mathbb{R}_{>0}$ and not just in $\mathfrak{R}_{>0}$, which is what makes this claim useful, but more difficult to prove.
        
        Let $S \subseteq \trop(W)$ be the finite set obtained by applying Proposition \ref{prop:non-archimedean-version} to $L$ and $W$, and let $z \in (\mathfrak{C}^\times)^n$. Then there are $\tau \in S$ and $z' \in (z \cdot \exp(L(\mathfrak{C})) \cdot (\mathcal{O}^\times)^n) \cap W(\mathfrak{C})$ such that $\alpha := \val(z') \in \tau$ and $\res(L) \times W_\tau$ is rotund. There exists $\tau' \in \trop(W)$ such that $\tau' \subseteq \tau$ and $\alpha \in \relint(\tau')$. By Lemma \ref{lem:rotundity-face}, also $\res(L) \times W_{\tau'}$ is rotund and so, after replacing $\tau$ by $\tau'$, we can assume without loss of generality that $\alpha \in \relint(\tau)$.

        Proposition \ref{prop:characterize-rotund}(\ref{two}) associates to each irreducible component $W_{\tau,\mathrm{Irr}}$ of $W_\tau$ such that $\res(L) \times W_{\tau,\mathrm{Irr}}$ is rotund a non-empty Zariski open subset, of which we will assume that it contains no points that also lie in another irreducible component of $W_\tau$; let $W_\tau^\circ$ be the union of all these sets. Note that the same set $W_\tau^\circ$ satisfies Condition (\ref{four}) of that proposition, as stated in its proof. By definition of $W_\tau^\circ$, we have that for every irreducible component $W_{\tau,\mathrm{Irr}}$ of $W_\tau$ either $\res(L) \times W_{\tau,\mathrm{Irr}}$ is not rotund or $\dim (W_{\tau,\mathrm{Irr}} \setminus W_\tau^\circ) < \dim W$ by Lemma \ref{lem:componentsofinitialform}. Therefore $\res(L) \times (W_\tau \setminus W_\tau^\circ)$ is not rotund. Since $z' \in z \cdot \exp(L(\mathfrak{C})) \cdot (\mathcal{O}^\times)^n$, there exists $\ell \in L(\mathfrak{C})$ such that $y = z' \cdot z^{-1} \cdot \exp(-\ell) \in (\mathcal{O}^\times)^n$ and so $\res(y) \in (\mathbb{C}^\times)^n.$
        We also deduce that
       \begin{equation}\label{eq:equal-translates}
           \frac{z^{-1} \cdot W(\mathfrak{C})}{\exp(L(\mathfrak{C}))} = \frac{z'^{-1} \cdot y \cdot W(\mathfrak{C})}{\exp(L(\mathfrak{C}))}.\end{equation}
        
        By Proposition \ref{prop:characterize-rotund}(\ref{eight}), there is a non-empty open subset $U$ of $\mathbb{S}_1^n(\mathbb{C})$ such that \[U \subseteq \frac{ \res(s(\alpha)z'^{-1}y) \cdot W_\tau(\mathbb{C})\ }{\exp(\res(L)(\mathbb{C}))}.\]
        
        Since $\res(L) \times (W_\tau \setminus W_\tau^\circ)$ is not rotund, Proposition \ref{prop:characterize-rotund}(\ref{three}), the upper semicontinuity of the fiber dimension of a holomorphic map \cite[Proposition 49.14]{Kaup_Kaup_1983}, and the Open Mapping Theorem together imply that, for every irreducible component $W_{\tau,\mathrm{Irr}}$ of $W_\tau \setminus W_\tau^\circ$ of dimension $\dim W_\tau$, all complex analytic components of all fibers of the map
        \[ \res(L)(\mathbb{C}) \times W_{\tau,\mathrm{Irr}}(\mathbb{C}) \to (\ctimes)^n, \quad (\ell,w) \mapsto w/\exp(\ell)\]
        are positive-dimensional.
        It then follows from \autocite[Proposition on p. 41 in Section 3.8 of Chapter 1]{Chirka_1989} that the set 
        \[\frac{ \res(s(\alpha)z'^{-1}y) \cdot (W_\tau \setminus W_\tau^\circ)(\mathbb{C})\ }{\exp(\res(L)(\mathbb{C}))}\] is contained in a countable union of complex analytic sets in $(\ctimes)^n$ of dimension $<n$. No non-empty open subset of $\mathbb{S}_1^n(\mathbb{C})$ is contained in the closure of a complex analytic set of dimension $<n$ (see \autocite[Proposition 6.10 and Corollary 6.11]{Gal23} for a proof), so by the Baire category theorem no non-empty open subset of $\mathbb{S}_1^n(\mathbb{C})$ is contained in a countable union of such sets either.
        
        Therefore, there exists some $s_0$ that lies in
        \[ U \cap \frac{ \res(s(\alpha)z'^{-1}y) \cdot W_\tau^\circ(\mathbb{C})\ }{\exp(\res(L)(\mathbb{C}))}.\]
        Since $W_\tau^\circ$ satisfies Condition (\ref{four}) of Proposition \ref{prop:characterize-rotund} for $\res(L) \times W_\tau$, there exists $\epsilon_0 \in \mathbb{R}_{>0}$ such that
        \[B(s_0,\epsilon_0) \subseteq \frac{ \res(s(\alpha)z'^{-1}y) \cdot W_\tau^\circ(\mathbb{C})\ }{\exp(\res(L)(\mathbb{C}))}.\] To ease notation, we now replace $W$ by $s_0^{-1} \cdot W$. This replaces $W_\tau$ and $W_\tau^\circ$ with $s_0^{-1} \cdot W_\tau$ and $s_0^{-1} \cdot W_\tau^\circ$ respectively, and $\trop(s_0^{-1} \cdot W)=\trop(W)$, so everything we have established so far carries over, with the only difference being that now we may take $s_0=1$. So there is $w_1 \in \res(s(\alpha)z'^{-1}y) \cdot W_\tau^\circ(\mathbb{C})$ such that $1 \in \frac{w_1}{\exp(\res(L))(\mathbb{C})}$, that is, such that $w_1 =\exp(\ell_1)$ for some $\ell_1\in \res(L)(\mathbb{C})$.

        By choosing $\epsilon_1 \in \mathbb{R}_{>0}$ small enough, we can make sure that $|w_1^{-1}|_\infty < \epsilon_1^{-1}$ and that, if $(\ell,w) \in (\res(L) \times \res(s(\alpha)z'^{-1}y) \cdot W_\tau)(\mathbb{C}) \cap \overline{B((\ell_1,w_1),\epsilon_1)}$ with $w = \exp(\ell)$, then $(\ell,w) = (\ell_1,w_1)$. After maybe decreasing $\epsilon_0$, we can also assume that
        \begin{equation}\label{eq:min} \min_{(\ell,w) \in (\res(L) \times \res(s(\alpha)z'^{-1}y) \cdot W_\tau)(\mathbb{C}) \cap \partial B((\ell_1,w_1),\epsilon_1)}{|w/\exp(\ell)-1|} > \epsilon_0.
        \end{equation}

        Recall that $L(\mathcal{O}) = L(\mathfrak{C}) \cap \mathcal{O}^n$. By Lemma \ref{lem:equal-residue} and its strengthening discussed in Remark \ref{rem:equal-residue} as well as Lemma \ref{lem:reduction-linear-subspace} and the definition of $\res(L)$, there exist $\ell_0 \in L(\mathcal{O})$ and $w_0 \in \left(z'^{-1} \cdot y \cdot W(\mathfrak{C})\right) \cap (\mathcal{O}^\times)^n$ such that $\res(\ell_0) = \ell_1$ and $\res(w_0) = w_1$. We want to apply Lemma \ref{lem:open-mapping-in-frakC} to deduce that
        \[B(w_0/\exp(\ell_0),\epsilon_0)(\mathfrak{C}) \subseteq \frac{z'^{-1} \cdot y \cdot W(\mathfrak{C})}{\exp(L(\mathfrak{C}))}.\]
        We have that $\res(|w_0^{-1}|_\infty) = |w_1^{-1}|_\infty < \epsilon_1^{-1}$ and so $|w_0^{-1}|_\infty < \epsilon_1^{-1}$.

        Suppose that $(\ell,w) \in (L \times z'^{-1} \cdot y \cdot W)(\mathfrak{C}) \cap \partial B((\ell_0,w_0),\epsilon_1)(\mathfrak{C})$. 
        Since $\res(|w_0^{-1}|_\infty) < \epsilon_1^{-1}$, we have that $\ell \in L(\mathcal{O})$ and $w \in (\mathcal{O}^\times)^n$. By Lemma \ref{lem:equal-residue} and its strengthening discussed in Remark \ref{rem:equal-residue} as well as Lemma \ref{lem:reduction-linear-subspace} and the definition of $\res(L)$, we also have that
        \[ (\res(\ell),\res(w)) \in (\res(L) \times \res(s(\alpha)z'^{-1}y) \cdot W_\tau)(\mathbb{C}) \cap \partial B((\ell_1,w_1),\epsilon_1).\]
        It then follows from  \eqref{eq:min} that
        \[\varepsilon_0 < |\res(w)/\exp(\res(\ell))-1| = \res(|w/\exp(\ell)-w_0/\exp(\ell_0)|)\]
        and so
        \[\varepsilon_0 < |w/\exp(\ell)-w_0/\exp(\ell_0)|.\]
        Hence, the hypothesis of Lemma \ref{lem:open-mapping-in-frakC} is satisfied and we deduce that
        \[B(w_0/\exp(\ell_0),\epsilon_0)(\mathfrak{C}) \subseteq \frac{z'^{-1} \cdot y \cdot W(\mathfrak{C})}{\exp(L(\mathfrak{C}))}.\]
        Finally, $\res(|w_0/\exp(\ell_0)-1|) = 0$ and so we deduce that
        \[B(1,\epsilon_0/2)(\mathfrak{C}) \subseteq \frac{z'^{-1} \cdot y \cdot W(\mathfrak{C})}{\exp(L(\mathfrak{C}))} \stackrel{\eqref{eq:equal-translates}}{=} \frac{z^{-1} \cdot W(\mathfrak{C})}{\exp(L(\mathfrak{C}))},\] which proves the claim from the beginning of this proof.
        
        We have proved that in $\mathfrak{R}$ the following first-order sentence holds $$\exists \epsilon >0 \left(\forall z \in \mathbb{G}_m^n \forall L \in G(d,n) \left(\exists s_1 \in \mathbb{S}_1^n \left(B(s_1,\epsilon) \subseteq \frac{z^{-1} \cdot W}{\exp(L)}\right)\right)\right)$$ because any positive infinitesimal element is a witness for it. Hence, this sentence must also hold true in $\mathbb{R}_{\exp,\sin}$, completing the proof.
    \end{proof}

    \begin{rem}
        If $W \subseteq \mathbb{G}_m^n$ is irreducible and not geometrically non-degenerate, then there is $L \leq \mathbb{G}_a^n$ defined over $\mathbb{Q}$ and of dimension $n - \dim W $ such that $\frac{W(\mathbb{C})}{\exp(L(\mathbb{C}))}$ has empty interior, so there is no $\epsilon$ as in Theorem \ref{thm:main-tropical}. Therefore the converse of Theorem \ref{thm:main-tropical} holds for irreducible subvarieties of $\mathbb{G}_m^n$, giving a characterization of geometrical non-degeneracy of a similar nature as the characterization of rotundity in Proposition \ref{prop:characterize-rotund}.
    \end{rem}

    \section{Equidistribution and proofs of the main theorems}\label{sec:equi}

     In this section, we will apply a well-known equidistribution theorem for Galois orbits of torsion points in algebraic tori (Theorem \ref{thm:torsion_points_equidistribute}) to prove Theorems \ref{thm:main} and \ref{thm:main-two}. 

    We denote the set of roots of unity in $\mathbb{C}^{\times}$ by $\mu_{\infty}$, and the algebraic closure of $\mathbb{Q}$ inside $\mathbb{C}$ by $\overline{\mathbb{Q}}$. For an element $z = (z_1,\hdots,z_n) \in (\mathbb{C}^\times)^n$ and a set $S$ of field automorphisms of $\mathbb{C}$, we define
    \[ S \cdot z = \{\sigma(z) \mid \sigma \in S\},\]
    where $\sigma(z)$ is defined to be
    \[ (\sigma(z_1),\hdots,\sigma(z_n)) \]
    for a field automorphism $\sigma$ of $\mathbb{C}$.

\begin{thm}\label{thm:torsion_points_equidistribute}
    Let $K \subseteq \mathbb{C}$ be a subfield that is finitely generated over $\mathbb{Q}$, let $n \in \mathbb{Z}_{>0}$, and let $B \subseteq (\mathbb{C}^\times)^n$ be some open Euclidean ball such that $B \cap \mathbb{S}^n_1(\mathbb{C}) \neq \emptyset$. 
    
  Let $(\zeta_j)_{j \in \mathbb{N}}$ be a sequence in $\mu_\infty^n$ such that for every algebraic subgroup $G \subsetneq \mathbb{G}^n_{m}$, the set of $j \in \mathbb{N}$ such that $\zeta_j \in G(\mathbb{C})$ is finite.
	
    Then there exists $N \in \mathbb{N}$ such that for all $j \in \mathbb{N}$: 
    \[ j \geq N \implies (\mathrm{Aut}(\mathbb{C}/K) \cdot \zeta_j) \cap B \cap \mathbb{S}^n_1(\mathbb{C}) \neq \emptyset.\]
\end{thm}

\begin{proof}
    Set $L = \overline{\mathbb{Q}} \cap K$. Then $L$ is a finite extension of $\mathbb{Q}$ by \cite[Theorem 24.9]{Isaacs_2009} and the restriction homomorphism $\mathrm{Aut}(\mathbb{C}/K) \to \mathrm{Gal}(\overline{\mathbb{Q}}/L)$ is surjective since $\mathrm{Gal}(\overline{\mathbb{Q}}K/K) \to \mathrm{Gal}(\overline{\mathbb{Q}}/L)$ is surjective by \cite[Chapter VI, Theorem 1.12]{Lang_Algebra} and $\mathrm{Aut}(\mathbb{C}/K) \to \mathrm{Gal}(\overline{\mathbb{Q}}K/K)$ is surjective as well. Hence, we can assume without loss of generality that $K$ is a finite extension of $\mathbb{Q}$.

    The theorem then follows from Bilu's equidistribution theorem; see \cite{Bilu_1997}, where the theorem is formulated over $\mathbb{Q}$, and see \cite{Kuehne_2022} for a proof of a related statement over an arbitrary number field.
	\end{proof}

\begin{lem}\label{lem:torsion_points_equidistribute}

    Let $K \subseteq \mathbb{C}$ be a subfield that is finitely generated over $\mathbb{Q}$, let $n \in \mathbb{Z}_{>0}$, and let $\epsilon > 0$. 
    
    Let $\mathscr{N}: \mu_{\infty}^n \to \mathbb{Z}_{> 0}$ be defined by
	\[ \mathscr{N}(x) = \min \{\lVert u \rVert \mid u \in \mathbb{Z}^n \backslash \{0\} \mbox{ such that } x^{u} = 1 \},\]
    where $\lVert u \rVert$ is the $L^1$-norm of $u$.
	
	 There exists $N = N(n,K,\epsilon) \in \mathbb{Z}_{>0}$ such that for every $\xi \in \mu_\infty^n$ with $\mathscr{N}(\xi) > N$, the set $\mathrm{Aut}(\mathbb{C}/K) \cdot \xi$ intersects every open Euclidean ball of radius $\epsilon$ centered at a point of $\mathbb{S}_1^n(\mathbb{C})$.
\end{lem}

\begin{proof}
  Since $\mathbb{S}_1^n(\mathbb{C})$ is compact, we can find finitely many open Euclidean balls $\widetilde{B}_1, \hdots, \widetilde{B}_M$ in $(\mathbb{C}^\times)^n$ centered at points of $\mathbb{S}_1^n(\mathbb{C})$ and of radius $\epsilon/2$ such that
    \[\mathbb{S}_1^n(\mathbb{C}) \subseteq \widetilde{B}_1 \cup \cdots \cup \widetilde{B}_M.\]

	We now argue by contradiction and assume that the lemma is false. Hence, there is a sequence $(\xi_j)_{j \in \mathbb{N}}$ in $\mu_\infty^n$ and a sequence of open Euclidean balls $(B_j)_{j \in \mathbb{N}}$ of radius $\epsilon$, each centered at some point of $\mathbb{S}_1^n(\mathbb{C})$, with the property that $\lim_{j \to \infty}{\mathscr{N}(\xi_j)} = \infty$ and
    \[ (\mathrm{Aut}(\mathbb{C}/K) \cdot \xi_j) \cap B_j = \emptyset\] for all $j \in \mathbb{N}$.
 
    For each $j \in \mathbb{N}$, there exists an $i(j) \in \{1,\hdots,M\}$ such that $\widetilde{B}_{i(j)}$ contains the center of $B_j$. It follows that $\widetilde{B}_{i(j)} \subseteq B_j$.

    We obtain that \emph{a fortiori} 
     \[ (\mathrm{Aut}(\mathbb{C}/K) \cdot \xi_j) \cap \widetilde{B}_{i(j)} = \emptyset\] for all $j \in \mathbb{N}$. After passing to a subsequence of $(\xi_j)_{j \in \mathbb{N}}$, we can assume that $i(j) = i_0$ for some $i_0 \in \{1,\hdots,M\}$ and all $j \in \mathbb{N}$.

     It follows that
    \[ (\mathrm{Aut}(\mathbb{C}/K) \cdot \xi_j) \cap \widetilde{B}_{i_0} = \emptyset\]
    for all $j \in \mathbb{N}$. At the same time, we deduce from $\lim_{j \to \infty}{\mathscr{N}(\xi_j)} = \infty$ that for every algebraic subgroup $G \subsetneq \mathbb{G}^n_{m}$, the set of $j \in \mathbb{N}$ such that $\xi_j \in G(\mathbb{C})$ is finite. We have found a contradiction with Theorem \ref{thm:torsion_points_equidistribute}, which finishes the proof.
	\end{proof}

\begin{lem}\label{lem:subgroups_equidistribute}
	Let $n \in \mathbb{Z}_{>0}$, let $K \subseteq \mathbb{C}$ be a subfield that is finitely generated over $\mathbb{Q}$, and let $\epsilon > 0$. There exists a finite set $\mathcal{G} = \{G_1,\hdots,G_M\}$, depending only on $n$, $K$, and $\epsilon$, such that $G_i \subsetneq \mathbb{G}^n_{m}$ is an algebraic subgroup for all $i = 1,\hdots,M$ and such that for every subtorus $J \subseteq \mathbb{G}^n_{m}$ and every torsion point $\zeta \in \mathbb{G}^n_{m}(\mathbb{C})$, one of the following holds:
	\begin{enumerate}
		\item[(i)] for every open Euclidean ball of radius $\epsilon$ centered at a point of $\mathbb{S}_1^n(\mathbb{C})$, the intersection of $\mathbb{S}_1^n(\mathbb{C}) \cap \sigma(\zeta) \cdot J(\mathbb{C})$ and the ball is not empty for some $\sigma \in \mathrm{Aut}(\mathbb{C}/K)$ or
		\item[(ii)] $\zeta \cdot J \subseteq G_i$ for some $i \in \{1,\hdots,M\}$.
	\end{enumerate}
\end{lem}

\begin{proof}
    Let $N = N(n,K,\epsilon)$ be the positive integer provided by Lemma \ref{lem:torsion_points_equidistribute}.

	Let $J$ be a subtorus of $\mathbb{G}^n_{m}$ and let $\zeta \in (\ctimes)^n$ be a torsion point. Since ($\zeta \cdot J(\mathbb{C})) \cap \mu_{\infty}^n$ is Zariski dense in $\zeta J$, it follows that either $\zeta J$ itself is contained in an algebraic subgroup defined by an equation $x^{u} = 1$ with $u \in \mathbb{Z}^n$ such that $0 < \lVert u \rVert \leq N$ or there is a point $\xi \in (\zeta \cdot J(\mathbb{C})) \cap \mu_{\infty}^n$ with $\mathscr{N}(\xi) > N$. In the second case, it follows from Lemma \ref{lem:torsion_points_equidistribute} that the first alternative in the conclusion of Lemma \ref{lem:subgroups_equidistribute} holds (note that the set $J(\mathbb{C})$ is invariant under any field automorphism of $\mathbb{C}$). We are now done by setting
	\begin{align*} \mathcal{G} = \{G \mid G \subseteq \mathbb{G}_m^n \mbox{ algebraic subgroup defined by an equation } \\
    x^{u} = 1 \mbox{ with $u \in \mathbb{Z}^n$ such that } 0 < \lVert u \rVert \leq N\}.
    \end{align*}
\end{proof}

\begin{proof}[Proof of Theorem \ref{thm:main}]
    By Theorem \ref{thm:main-tropical}, there exists $\epsilon > 0$ such that for every $z \in (\mathbb{C}^{\times})^n$ and every subtorus $H \subseteq \mathbb{G}^n_{m}$ with $\dim H + \dim W \geq n$ (in Theorem \ref{thm:main-tropical}, the dimensions add up to $n$, but this is no restriction since one can apply Theorem \ref{thm:main-tropical} with $L = LH'$ for any subtorus $H'$ of $H$ of dimension $n -\dim W$), the image of the map
	\[ \psi_{W,H,z}: W(\mathbb{C}) \times H(\mathbb{C}) \to (\mathbb{C}^{\times})^n, \quad (w,y) \mapsto wyz\]
	contains an open Euclidean ball of radius $\epsilon$ centered at a point of $\mathbb{S}_1^n(\mathbb{C})$.
	
	Let $z \in (\mathbb{C}^{\times})^n$ be an arbitrary point  and let $H \subseteq \mathbb{G}_m^n$ be a subtorus of dimension at least $n-\dim W$. Note that
	\begin{align*} W(\mathbb{C}) \cap z \cdot H(\mathbb{C}) \neq \emptyset &\Leftrightarrow z \in \psi_{W,H,1}(W(\mathbb{C}) \times H(\mathbb{C}))\\
    &\Leftrightarrow 1 \in \psi_{W,H,z^{-1}}(W(\mathbb{C}) \times H(\mathbb{C})).
    \end{align*}
	We know that the image of $\psi_{W,H,z^{-1}}$ contains an open Euclidean ball of radius $\epsilon$ centered at a point of $\mathbb{S}_1^n(\mathbb{C})$. Since
	\[ H(\mathbb{C}) \cdot \psi_{W,H,z^{-1}}(W(\mathbb{C}) \times H(\mathbb{C})) = \psi_{W,H,z^{-1}}(W(\mathbb{C}) \times H(\mathbb{C})),\]
	the theorem follows from Lemma \ref{lem:subgroups_equidistribute} applied with $\zeta = (1,\hdots,1)$ and $K = \mathbb{Q}$.
\end{proof}

\begin{proof}[Proof of Theorem \ref{thm:main-two}]

    By Theorem \ref{thm:main-tropical}, there exists $\epsilon > 0$ such that for every $z \in (\mathbb{C}^{\times})^n$ and every subtorus $H \subseteq \mathbb{G}^n_m$ with $\dim H + \dim W \geq n$ (again, we can apply Theorem \ref{thm:main-tropical} with $L = LH'$ for any subtorus $H'$ of $H$ of dimension $n -\dim W$), the image of the map
	\[ \psi_{W,H,z}: W(\mathbb{C}) \times H(\mathbb{C}) \to (\mathbb{C}^{\times})^n, \quad (w,y) \mapsto wyz\]
	contains an open Euclidean ball of radius $\epsilon$ centered at a point of $\mathbb{S}_1^n(\mathbb{C})$.

    Let now $H$ be a fixed subtorus of $\mathbb{G}^n_{m}$ such that $\dim H + \dim W \geq n$ and let $\zeta \in (\ctimes)^n$ be a torsion point. By the above, the image of $\psi_{W,H,1}$, which is equal to $W(\mathbb{C}) \cdot H(\mathbb{C})$, contains an open Euclidean ball of radius $\epsilon$ centered at a point of $\mathbb{S}_1^n(\mathbb{C})$.
	
    The subvariety $W$ of $\mathbb{G}^n_m$ can be defined by equations with coefficients in some field $K \subseteq \mathbb{C}$ that is finitely generated over $\mathbb{Q}$. It follows from Lemma \ref{lem:subgroups_equidistribute} that either $\zeta \cdot H$ is contained in one of finitely many algebraic subgroups $G \subsetneq \mathbb{G}^n_{m}$, depending only on $n$, $K$, and $\epsilon$, or $(W(\mathbb{C}) \cdot H(\mathbb{C})) \cap (\sigma(\zeta) \cdot H(\mathbb{C})) \neq \emptyset$ for some $\sigma \in \mathrm{Aut}(\mathbb{C}/K)$. We can assume without loss of generality that the second alternative holds and hence $W(\mathbb{C}) \cap \sigma(\zeta) \cdot H(\mathbb{C}) \neq \emptyset$.

    But since $W$ is defined over $K$, the set $W(\mathbb{C})$ is invariant under $\sigma^{-1}$ (and the same holds for $H(\mathbb{C}$)). It follows that $W(\mathbb{C}) \cap \zeta \cdot H(\mathbb{C}) \neq \emptyset$ and we are done.
\end{proof}

\section{Applications}\label{sec:appli}
We start this section by proving a version of Theorem \ref{thm:main}  over arbitrary algebraically closed fields of characteristic 0. The second-named author thanks Vahagn Aslanyan and Vincenzo Mantova for pointing this out.

\begin{thm}\label{thm:main-acf}
    Let $K$ be an algebraically closed field of characteristic $0$ and $W \subseteq \mathbb{G}_{m,K}^n$ a geometrically non-degenerate algebraic subvariety. 

    Then there exists a finite set $\mathcal{H}=\{H_1,\dots,H_N\}$ such that $H_i \subsetneq \mathbb{G}^n_{m,K}$ is a subtorus for all $H_i \in \mathcal{H}$ and such that for every subtorus $H \subseteq \mathbb{G}_{m,K}^n$ with $\dim H + \dim W \geq n$, one of the following holds:
    \begin{itemize}
        \item[(i)] For every $z \in (K^\times)^n$, $W(K) \cap z \cdot H(K) \neq \emptyset$ or
        \item[(ii)] $H \subseteq H_i$ for some $H_i \in \mathcal{H}$.
    \end{itemize}
\end{thm}

\begin{proof}
    Let $K_0$ be a countable algebraically closed subfield of $K$ such that $W$ is defined over $K_0$. Let $\iota:K_0 \hookrightarrow \mathbb{C}$ be an embedding, so that we may see $W$ as defined over $\mathbb{C}$. We may then apply Theorem \ref{thm:main} and obtain a finite set $\mathcal{H}=\{H_1,\dots,H_N\}$ of subtori. Let $H \subseteq \mathbb{G}_{m,K}^n$ be a subtorus with $\dim H + \dim W \geq n$, let $z \in (K^\times)^n$, and assume that $H \nsubseteq H_i$ for every $H_i \in \mathcal{H}$. Note that $H$ and all of the $H_i$ are defined over $\mathbb{Q}$. Extend the embedding $\iota$ to an embedding $\iota'$ of $K_0(z)$ into $\mathbb{C}$. By Theorem \ref{thm:main}, $W(\mathbb{C}) \cap \iota'(z) \cdot H(\mathbb{C}) \neq \emptyset$. We may then conclude using the fact that $K$ is algebraically closed and that $\iota'$ is an embedding.
\end{proof}

A similar argument allows us to generalize also Theorem \ref{thm:main-two} to arbitrary algebraically closed fields of characteristic $0$.

We now obtain some applications of Theorems \ref{thm:main} and \ref{thm:main-two}, which in view of what we just showed are valid over arbitrary algebraically closed fields of characteristic $0$. First we show that results on likely intersections allow us to give a new proof of the Manin--Mumford conjecture for algebraic tori.

\begin{prop}\label{prop:proof-manin-mumford}
Let $K$ be an algebraically closed field of characteristic $0$ and let $V$ be a subvariety of $\mathbb{G}^n_{m,K}$. Then the Zariski closure of the set of torsion points that lie on $V$ is a finite union of torsion cosets.
\end{prop}

\begin{proof}
Any subvariety of $\mathbb{G}^n_{m,K}$ is a finite union of irreducible subvarieties of $\mathbb{G}^n_{m,K}$, so we can assume that $V$ is irreducible.

Any irreducible subvariety $V$ of $\mathbb{G}^n_{m,K}$ is an intersection of finitely many irreducible hypersurfaces $V_1, \hdots, V_k$ in $\mathbb{G}^n_{m,K}$. Suppose that the statement of the proposition holds for every $V_i$. The set of torsion points that lie on $V$ is then contained in a finite union $T_i$ of torsion cosets that are all contained in $V_i$ ($i = 1,\hdots, k$). It follows that the set of torsion points that lie on $V$ is contained in $T := \bigcap_{i=1}^{k}{T_i}$. Since the intersection of two torsion cosets is a finite union of torsion cosets, $T$ is again a finite union of torsion cosets and $T$ is contained in $V$. Since every torsion coset contains a Zariski dense set of torsion points, the Zariski closure of the set of torsion points that lie on $V$ is equal to $T$ and we are done. We can and will therefore assume that $V$ is an irreducible hypersurface. It suffices to show that there is a finite list of torsion cosets contained in $V$ such that each torsion point that lies on $V$ belongs to one of the torsion cosets in the list.

We then induct on $n$, the case $n = 1$ being trivial.

If $n \geq 2$ and the stabilizer of $V$ is positive-dimensional, then we can find a surjective homomorphism $\varphi: \mathbb{G}^n_{m,K} \to \mathbb{G}^k_{m,K}$ for some $k \in \{1,\hdots,n-1\}$ such that $V' = \varphi(V)$ is closed in $\mathbb{G}^k_{m,K}$ and $V = \varphi^{-1}(V')$. We can then apply the inductive hypothesis for $V'$ and are done. Hence, we can assume that the stabilizer of $V$ is finite.

The hypersurface $V$ is defined by some polynomial equation
\[ f(x_1,\hdots,x_n) = 0.\]

We now consider the hypersurface $V'$ in $\mathbb{G}^{n+1}_{m,K}$ that is defined by the equation
\[ f(x_1,\hdots,x_n) = x_{n+1}.\]
Clearly, $V'$ is irreducible. We denote by $\pi: \mathbb{G}^{n+1}_{m,K} \to \mathbb{G}^{n}_{m,K}$ the projection to the first $n$ coordinates, then $\pi(V')$ is the complement of $V$. It follows that $\pi$ maps the stabilizer of $V'$ into the stabilizer of $V$, which is finite. Furthermore, if $(z_1,\hdots,z_{n+1})$ belongs to the stabilizer of $V'$, then
\[ z_{n+1} = f(z_1x_1,\hdots,z_nx_n)f(x_1,\hdots,x_n)^{-1},\]
where $(x_1,\hdots,x_n) \in (K^\times)^n$ is an arbitrary point at which $f$ does not vanish. Hence, the restriction of $\pi$ to the stabilizer of $V'$ is injective. We deduce that the stabilizer of $V'$ is finite and so $V'$ is geometrically non-degenerate.

Suppose now that $(\zeta_1,\hdots,\zeta_n) \in V(K)$ is a torsion point. It follows that the intersection of $V'$ with the torsion coset
\[ \{(\zeta_1,\hdots,\zeta_n)\} \times \mathbb{G}_{m,K}\]
is empty. Theorem \ref{thm:main-two} (or more precisely its version over $K$) now yields a finite set of proper torsion cosets in $\mathbb{G}^{n+1}_{m,K}$ such that $\{(\zeta_1,\hdots,\zeta_n)\} \times \mathbb{G}_{m,K}$ is contained in one of them or, equivalently, a finite set of proper torsion cosets in $\mathbb{G}^n_{m,K}$ such that $(\zeta_1,\hdots,\zeta_n)$ belongs to one of them.

Let $\zeta H$ be such a torsion coset, where $\zeta$ is a torsion point and $H$ is a proper subtorus of $\mathbb{G}^n_{m,K}$. If $\zeta H \subseteq V$, we add $\zeta H$ to our (finite) list of torsion cosets contained in $V$ and containing all torsion points that lie on $V$ and we are done. Otherwise, $\zeta H \cap V$ is a non-empty union of irreducible hypersurfaces inside $\zeta H$. Recalling that $H$ is isomorphic as an algebraic group to $\mathbb{G}^k_{m,K}$ for some $k \in \{1,\hdots,n-1\}$ and using that translation by $\zeta^{\pm 1}$ maps torsion cosets onto torsion cosets, we are then again done by induction.
\end{proof}

We next prove Corollary \ref{cor:curvecase}.

\begin{proof}[Proof of Corollary \ref{cor:curvecase}]
Suppose that $W(K) \cap z \cdot H(K) = \emptyset$ for some subtorus $H \subseteq \mathbb{G}^{n}_{m,K}$ of dimension $n-1$ and some $z \in (K^\times)^n$. By Theorem \ref{thm:main-acf}, we then have that $H$ is contained in and therefore equal to one of finitely many proper subtori. It remains to prove that there are only finitely many possibilities for the coset $z \cdot H$. But the restriction of the quotient homomorphism $\pi_H: \mathbb{G}^{n}_{m,K} \to \mathbb{G}^{n}_{m,K}/H$ to $W$ is not constant and hence it is dominant since its domain and its codomain are both curves. It follows that $\pi_H(W)$ is constructible and dense in $\mathbb{G}^n_{m,K}/H$ and therefore $(\mathbb{G}^n_{m,K}/H) \setminus \pi_H(W)$ is finite. It follows that $z \cdot H = \pi_H^{-1}(y)$ for one of finitely many $y \in (\mathbb{G}^n_{m,K}/H)(K)$ and we are done.
\end{proof}

We conclude with the proof of Theorem \ref{thm:probabilistic}.

\begin{proof}[Proof of Theorem \ref{thm:probabilistic}]
For a vector $a = (a_{i,j})_{i = 1,\hdots,d;~j = 1,\hdots,n} \in \mathbb{Z}^{nd}$, we define $H_a \subseteq \mathbb{G}^n_{m,K}$ by the system of equations
\[ \prod_{j=1}^{n}{x_j^{a_{i,j}}} = 1~(i = 1, \hdots, d).\]
Then, $H_a$ is an algebraic subgroup of $\mathbb{G}^n_{m,K}$. We denote by $H_a^0$ the connected component of $H_a$ that contains the neutral element.

If the vectors $a_i = (a_{i,j})_j \in \mathbb{Z}^n$ ($i = 1, \hdots, d$) are linearly independent, then for each $y = (y_1,\hdots,y_d) \in (K^\times)^d$, there exists $z = (z_1,\hdots,z_n) \in (K^\times)^n$ such that
\[ \prod_{j=1}^{n}{z_j^{a_{i,j}}} = y_i \quad (i = 1, \hdots, d).\]
We can then apply Theorem \ref{thm:main-acf} to the intersection $W(K) \cap z \cdot H_a^0(K)$. We find that this intersection is non-empty unless $H_a^0$ is contained in one of finitely many proper subtori (depending only on $W$) or, equivalently, one of finitely many non-zero vectors $b_1, \hdots, b_k \in \mathbb{Z}^n$ is contained in the vector subspace of $\mathbb{Q}^n$ that is generated by the $a_i$ ($i = 1,\hdots, d$).

Therefore, it is enough to show that, for a given $b \in \mathbb{Z}^n \backslash \{0\}$, we have that
\[ \lim_{N \to \infty}{\frac{\#\mathcal{T}(n,d,b,N)}{(2N+1)^{nd}}} = 0,\]
where $\mathcal{T}(n,d,b,N)$ is the set of tuples of vectors $(a_1,\hdots,a_d) \in \mathbb{Z}^{nd}$ such that their entries are bounded in absolute value by $N$ and the tuple $(a_1, \hdots, a_d, b)$ is linearly dependent.

Since $b = (b_1, \hdots, b_n) \neq 0$, there exists some $j_0$ such that $b_{j_0} \neq 0$. Given $(a_1,\hdots,a_d) \in \mathcal{T}(n,d,b,N)$, where $a_i = (a_{i,1},\hdots,a_{i,n})$ ($i = 1,\hdots,d$), we can then find $k \in \{1, \hdots, d\}$ and pairs of indices $(i_1,j_1), \hdots, (i_k,j_k)$ such that
\[
\#\{j_0, \hdots,j_k\} = \#\{i_1,\hdots,i_k\} + 1 = k + 1
\]
and
\[ \det \begin{pmatrix}
a_{i_1,j_0} & a_{i_1,j_1} & \hdots & a_{i_1,j_k}\\
\vdots &  &  & \vdots\\
a_{i_k,j_0} & a_{i_k,j_1} & \hdots & a_{i_k,j_k}\\
b_{j_0} & b_{j_1} & \hdots & b_{j_k}\\
\end{pmatrix} = 0,\]
but
\[ \det \begin{pmatrix}
a_{i_1,j_0} & a_{i_1,j_1} & \hdots & a_{i_1,j_{k-1}}\\
\vdots &  &  & \vdots\\
a_{i_{k-1},j_0} & a_{i_{k-1},j_1} & \hdots & a_{i_{k-1},j_{k-1}}\\
b_{j_0} & b_{j_1} & \hdots & b_{j_{k-1}}\\
\end{pmatrix} \neq 0.\]
By expanding the first determinant along the last column, we deduce from this that $a_{i_k,j_k}$ is uniquely determined by the other $a_{i,j}$ and so the number of possibilities for $(a_1,\hdots,a_d)$ is at most $(2N+1)^{nd-1}$. This suffices to conclude.
\end{proof}

\section*{Acknowledgements}
The first-named author thanks Francesco Campagna, who contributed to this project in its early stages and, together with the first-named author, proved Theorems \ref{thm:main} and \ref{thm:main-two} in the special case where $\dim W  = 1$, using a different method relying on the Mason-Stothers theorem, Martin Orr for pointing him towards results on likely intersections in $Y(1)^2$, and Thomas Scanlon for helpful discussions. The second-named author thanks Vahagn Aslanyan and Vincenzo Mantova for their suggestions on Theorem \ref{thm:main-acf}, and the Department of Mathematics of the University of Padova for their hospitality in Summer 2024 as this paper was being written. Both authors thank Harry Schmidt for suggesting this collaboration.

\vspace{\baselineskip}
\noindent
\framebox[\textwidth]{
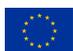
	\begin{tabular*}{0.96\textwidth}{@{\extracolsep{\fill} }cp{0.84\textwidth}}
		\raisebox{-0.7\height}{
			\begin{tikzpicture}[y=0.80pt, x=0.8pt, yscale=-1, inner sep=0pt, outer sep=0pt, 
			scale=0.12]
			\definecolor{c003399}{RGB}{0,51,153}
			\definecolor{cffcc00}{RGB}{255,204,0}
			\begin{scope}[shift={(0,-872.36218)}]
			\path[shift={(0,872.36218)},fill=c003399,nonzero rule] (0.0000,0.0000) rectangle (270.0000,180.0000);
			\foreach \myshift in 
			{(0,812.36218), (0,932.36218), 
				(60.0,872.36218), (-60.0,872.36218), 
				(30.0,820.36218), (-30.0,820.36218),
				(30.0,924.36218), (-30.0,924.36218),
				(-52.0,842.36218), (52.0,842.36218), 
				(52.0,902.36218), (-52.0,902.36218)}
			\path[shift=\myshift,fill=cffcc00,nonzero rule] (135.0000,80.0000) -- (137.2453,86.9096) -- (144.5106,86.9098) -- (138.6330,91.1804) -- (140.8778,98.0902) -- (135.0000,93.8200) -- (129.1222,98.0902) -- (131.3670,91.1804) -- (125.4894,86.9098) -- (132.7547,86.9096) -- cycle;
			\end{scope}
			\end{tikzpicture}
		}
		&
		Gabriel Dill has received funding from the European Research Council (ERC) under the European Union's Horizon 2020 research and innovation programme (grant agreement n$^\circ$ 945714).
	\end{tabular*}
}
\smallskip

This material is based upon work supported by the National Science Foundation under Grant No. DMS--1928930 while the first-named author was in residence at the Mathematical Sciences Research Institute in Berkeley, California, during the Spring 2023 semester. The first-named author also thanks the DFG for its support (grant no. EXC-2047/1 - 390685813). The second-named author was partially supported by the program GeoMod ANR-19-CE40-0022-01 (ANR-DFG) while in Freiburg and by the PRIN 2022-Modelli, Insiemi e Classificazioni, prot. 2022TECZJA while in Pisa.

\printbibliography

\end{document}